\documentclass{fundam}

\usepackage{amsmath,amsfonts,amssymb, mathtools, latexsym, stmaryrd}
\usepackage[T1]{fontenc}
\usepackage{mathabx}
\usepackage[all,color]{xy}
\usepackage{rotating}
\usepackage{enumitem}

\usepackage{bbm}

\newcommand{\vdim}{\underline{\mathbf{inc}}}
\newcommand{\tr}{\mathbf{tr}}
\newcommand{\Dyn}{\mathbf{Dyn}}
\newcommand{\Z}{\mathbb{Z}}

\newcommand{\Inc}{\mathbf{Inc}}
\newcommand{\A}{\mathbb{A}}

\newcommand{\D}{\mathbb{D}}
\newcommand{\E}{\mathbb{E}}

\newcommand{\CRnk}{\mathbf{cork}}
\newcommand{\Id}{\mathbf{I}}
\newcommand{\Quad}{\mathbf{UQuad}}
\newcommand{\Quiv}{\mathbf{UQuiv}}
\newcommand{\Adj}{\widecheck{\mathrm{Adj}}}
\newcommand{\SAdj}{\mathrm{Adj}}

\newcommand{\Star}{\mathbb{S}}
\newcommand{\bas}{\mathbf{e}}

\newcommand{\sou}{\mathbf{s}}
\newcommand{\tar}{\mathbf{t}}
\newcommand{\Cox}{\mathrm{\Phi}}
\newcommand{\cox}{\mathrm{\varphi}}
\newcommand{\coxN}{\mathbf{c}}
\newcommand{\RcoxN}{\mathbf{C}_{re}}
\newcommand{\va}{\lambda}

\newcommand{\CSpec}{\mathbf{CSpec}}
\newcommand{\MiNu}{\mathbf{v}}
\newcommand{\LCox}{\Lambda}
\newcommand{\chr}{\mathrm{char}}
\newcommand{\Part}{\mathcal{P}}
\newcommand{\ct}{\mathbf{ct}}
\newcommand{\lcm}{\mathbf{lcm}}
\newcommand{\bulito}{\mathmiddlescript{\bullet}}
\newcommand{\dd}{\mathbbm{d}}

\newcommand\mathmiddlescript[1]{\vcenter{\hbox{$\scriptstyle #1$}}}

\newtheorem{algorithm}{Algorithm}

   \usepackage{hyperref}

\begin{document}
	
\setcounter{page}{221}
\publyear{22}
\papernumber{2109}
\volume{185}
\issue{3}

   \finalVersionForARXIV

\title{Coxeter Invariants for Non-negative Unit Forms of Dynkin Type $\A_{r}$}

\author{Jes\'us Arturo~Jim\'enez Gonz\'alez\thanks{Address of correspondence: Instituto de Matem\'aticas, Mexico City, Mexico. \newline \newline
          \vspace*{-6mm}{\scriptsize{Received June 2021; \ accepted Febryary 2022.}}}
\\
 Instituto de Matem\'aticas, UNAM, Mexico\\
jejim@im.unam.mx
}

\maketitle

\runninghead{J.A. Jim\'enez Gonz\'alez}{Coxeter Invariants for Non-negative Unit Forms of Dynkin Type $\A_{r}$}

\vspace*{-5mm}
\begin{abstract}
Two integral quadratic unit forms are called strongly Gram congruent if their upper triangular Gram matrices are $\Z$-congruent. The paper gives a combinatorial strong Gram invariant for those unit forms that are non-negative of Dynkin type $\A_{r}$ (for $r \geq 1$), within the framework introduced in~[Fundamenta Informaticae~\textbf{184}(1):49--82, 2021], and uses it to determine all corresponding Coxeter polynomials and (reduced) Coxeter numbers.

\end{abstract}

\section{Introduction} \label{(S):Zero}

\subsection*{Basic notions.}

Integral quadratic forms (that is, homogeneous polynomials of degree two with integer coefficients), have been a central topic of study, sometimes indirectly, in many areas of abstract algebra and graph theory (see for instance~\cite{kB83,cmR,BKL06,BR07,BPS11} and the introductory notes of~\cite{BJP19,BGZ06,CDD21}). One approach, initiated by Simson in~\cite{dS11a,dS13} and developed intensively by Simson and collaborators, see for instance~\cite{KS15a,KS15b,KS15c,dS16a,SZ17,dS18,dS20,dS21a,dS21b}, consists in substituting $q$ by the (upper triangular) standard morsification $b_q:\Z^n \times \Z^n \to \Z$ of $q$, and focuses on the Coxeter formalism of $b_q$. In this way, one defines the strong Gram congruence among unit forms, and attaches the Coxeter invariants of $b_q$ to $q$, which are also strong Gram invariants of $q$ (cf.~\cite[Lemma~1.3]{dS21a} or~\cite[Lemma~4.6]{jaJ2020a}), a point of view mainly motivated from the Auslander-Reiten theory of associative algebras, see~\cite{dS11a,dS13} and references therein.

The aim of this paper is to explicitly determine some classical Coxeter invariants, namely the Coxeter polynomial and the (reduced) Coxeter number, associated to a connected non-negative unit form of Dynkin type $\A_{r}$ ($r \geq 1$).  Work in this direction may be found, for instance, in~\cite[Theorem~3.3]{dS13}, \cite[Theorem~3.2]{dS13a}, \cite[Theorem~5.1]{KS15b} and~\cite[Theorem~2.3]{FS13b} for small number of variables, in~\cite[Theorem~2.4]{dS21a} and \cite[Theorem~2.2]{dS21b} for the Dynkin types $\A$ and $\D$ of corank zero, respectively, and in~\cite[Theorem~1.10]{GSZ14} and~\cite[Corollary~11]{SZ12} for quadratic forms associated to principal posets (corank one). We follow the graph theoretical technique introduced in~\cite{jaJ2018}, and applied recently to the study of the strong Gram congruence in~\cite{jaJ2020a}. We refer the reader to the introduction of~\cite{jaJ2020a} for some historical remarks and further references for these methods.

\medskip
Throughout the paper we identify an integral quadratic form $q:\Z^n\! \to\! \Z$, $q(x)\!=\!\sum_{1 \leq i \leq j \leq n}q_{i,j}x_ix_j$, with the upper triangular matrix $\widecheck{G}_q=(g_{i,j})$ given by $g_{i,j}=q_{i,j}$ if $1 \leq i \leq j \leq n$ and $g_{i,j}=0$ if $1 \leq j<i \leq n$. This matrix is referred to as (non-symmetric or upper triangular) \textbf{Gram matrix} of $q$, and is the unique upper triangular integer matrix satisfying
\[
q(x)=x^{\tr}\widecheck{G}_qx, \quad \text{for any column vector $x$ in $\Z^n$}.
\]
The quadratic form $q$ is \textbf{unitary} (or a \textbf{unit form}) if all diagonal entries of $\widecheck{G}_q$ are equal to $1$. The \textbf{symmetric Gram matrix} $G_q$ of $q$ is given by $G_q=\widecheck{G}_q+\widecheck{G}_q^{\tr}$. Two unit forms $q$ and $q'$ are called \textbf{weakly} (resp. \textbf{strongly}) \textbf{Gram congruent}, if there is a $\Z$-invertible $n \times n$ matrix $B$ such that $G_{q'}=B^{\tr}G_qB$ (resp. $\widecheck{G}_{q'}=B^{\tr}\widecheck{G}_qB$), written $q' \sim^B q$ or $q' \sim q$ (resp. $q' \approx^B q$ or $q' \approx q$).  In what follows we use standard notions and results on quadratic forms, such as positivity, non-negativity, corank, connectedness and Dynkin type (cf.~\cite{BP99,SZ17,BJP19,dS20}).  For instance, the weak Gram classification of non-negative unit forms, achieved in~\cite{BP99} and~\cite{dS16a} with different methods, assigns a unique Dynkin type $\A_n$, $\D_m$ or $\E_p$ (for $n \geq 1$, $m \geq 4$ or $p \in \{6,7,8\}$) and a non-negative corank to any weak congruence class of non-negative unit forms. Since, clearly, the strong Gram congruence refines the weak one, it is natural to approach the strong classification problem by the cases of the weak classification. Here we continue the study of the strong Gram congruence among non-negative unit forms of Dynkin type $\A_{r}$ started in~\cite{jaJ2020a}. For convenience, we present relevant definitions and constructions of~\cite{jaJ2020a} as needed.

\medskip
Note that if $q$ is a connected non-negative unit form in $n \geq 1$ variables, then the corank $c$ of $q$ is smaller than $n$, since $n-c$ is the rank of $q$. For $0 \leq c <n$, we denote by $\Quad^c_{\A}(n)$ the set of connected non-negative unit forms in $n$ vertices, with corank $c$ and Dynkin type $\A_{n-c}$,
\[
\Quad^c_{\A}(n)=\{q:\Z^n \to \Z \mid \text{$q$ is connected, $q \geq 0$, $\Dyn(q)=\A_{n-c}$ and $\CRnk(q)=c$}\}.
\]

\subsection*{Partitions and permutations.}

A \textbf{partition} $\pi$ of an integer $m \geq 1$ (written $\pi \vdash m$) is a non-increasing sequence of positive integers $\pi=(\pi_1,\ldots,\pi_{\ell(\pi)})$ for some $\ell(\pi)\geq 1$, such that $m=\sum_{a=1}^{\ell(\pi)}\pi_a$. The integer $\ell(\pi)$ is called \textbf{length} or \textbf{number of parts} of $\pi$. For instance, let $\rho$ be a \textbf{permutation} of the set $\{1,\ldots,m\}$. The orbits of $\rho$ determine a set-partition
\[
\{1,\ldots,m\}=\mathcal{P}_1 \sqcup \ldots \sqcup \mathcal{P}_{\ell},
\]
for some $\ell \geq 1$ (that is, two indices $v,v' \in \{1,\ldots,m\}$ belong to the same subset $\mathcal{P}_r$ if and only if there is $t \geq 0$ such that $v'=\rho^t(v)$, and for any element $v$ there is a subset $\mathcal{P}_r$ containing $v$). The sequence of cardinalities of $\mathcal{P}_1,\ldots,\mathcal{P}_{\ell}$, ordered non-increasingly, is a partition $\pi(\rho)$ of $m$, usually called the \textbf{cycle type} (or cycle structure) of the permutation $\rho$ (cf.~\cite[\S 2.2]{jR06}). Conversely, for any partition $\pi=(\pi_1,\ldots,\pi_{\ell})$ of $m$, denote by $\rho_{\pi}$ the permutation of $\{1,\ldots,m\}$ given as composition of cycles of length $\pi_r$,
\[
\rho_{\pi}=(1,\ldots,\pi_1)(\pi_1+1,\ldots,\pi_1+\pi_2)\cdots (\pi_1+\ldots+\pi_{\ell-1}+1,\ldots,m-1,m).
\]
Clearly, $\pi(\rho_{\pi})=\pi$. It is well known that for two permutations $\rho$ and $\rho'$ of the set $\{1,\ldots,m\}$, we have $\pi(\rho)=\pi(\rho')$ if and only if $\rho$ and $\rho'$ are conjugated permutations (that is, $\rho'=\xi\rho\xi^{-1}$ for some permutation $\xi$ of $\{1,\ldots,m\}$, see for instance~\cite[Proposition~2.33]{jR06}). In particular, if $P(\rho)$ denotes the matrix with $P(\rho)\bas_v=\bas_{\rho(v)}$ for any $v \in \{1,\ldots,m\}$ where $\bas_v$ is the $v$-th canonical vector of $\Z^m$ ($P(\rho)$ is called the \textbf{permutation matrix} of $\rho$), then the characteristic polynomial $\chr_{P(\rho)}(\va)$ of $P(\rho)$ only depends on the cycle type $\pi(\rho)$ of $\rho$. Since the characteristic polynomial of the permutation matrix of a cycle of length $r$ is $(\va^r-1)$, if $\pi(\rho)=(\pi_1,\ldots,\pi_{\ell})$ then the characteristic polynomial of $P(\rho)$ is
\[
\chr_{P(\rho)}(\va)=\prod_{a=1}^{\ell(\pi)}(\va^{\pi_a}-1).
\]
Define the \textbf{characteristic polynomial} of a partition $\pi$ as the characteristic polynomial of the permutation matrix $P(\rho_{\pi})$, that is, $\chr_{\pi}(\va):=\chr_{P(\rho_{\pi})}(\va)$. By the comments above, $\chr_{\pi}(\va)$ is the characteristic polynomial of the permutation matrix of any permutation with cycle type $\pi$.

\medskip
For arbitrary $c \geq 0$ and $m \geq 1$, we consider the set of partitions of the integer $m$ having their number of parts restricted by $c$ as follows,
\[
\Part_1^{c}(m)=\{ \pi \vdash m \mid 0 \leq c-(\ell(\pi)-1) \equiv 0 \mod 2 \}.
\]

\subsection*{Overview of the paper.}

The \textbf{Coxeter matrix} $\Cox_q$ of a unit form $q$ is given by $\Cox_q=-\widecheck{G}^{\tr}_q\widecheck{G}^{-1}_q$ (compare with more usual definitions as given in~\cite{BJP19} or~\cite{dS16a}). The characteristic polynomial of $\Cox_q$ is called \textbf{Coxeter polynomial} of $q$, and is denoted by $\cox_q(\va)$. It is well known that if $q' \approx q$ for a unit form $q'$, then $\cox_{q'}(\va)=\cox_q(\va)$ (see for instance~\cite[Lemma~4.6]{jaJ2020a}). Our goal is to prove the following result (see Theorem~\ref{MAIN} below).

\subsection*{Main theorem.}

For any integers $0 \leq c < n$, there is a surjective function
\[
\xymatrix{ \Quad_{\A}^c(n) \ar[r]^-{\ct} & \Part^c_1(n-c+1),}
\]
which is invariant under strong Gram congruence, and such that for any $q$ in $\Quad_{\A}^c(n)$,
\[
\cox_q(\va)=(\va-1)^{c-1}\chr_{\ct(q)}(\va).
\]

The partition $\ct(q)$ will be referred to as \textbf{cycle type} of $q$, and the induced function on the quotient $\Quad_{\A}^c(n)/\approx$ will be also denoted by $\ct$.  Although the constructions leading to the proof of the Main Theorem are straightforward, most of the preparatory arguments are fairly technical. For convenience, we sketch the steps of the proof (see precise definitions below).
\begin{itemize}
\itemsep=0.95pt
 \item[i)] For a quiver $Q$ with $m$ vertices and $n$ arrows, and vertex-arrow incidence matrix $I(Q)$, consider the quadratic form $q_Q:\Z^n \to \Z$ given by
\[
  q_Q(x)=\frac{1}{2}||I(Q)x||^2.
\]
It is shown in~\cite{jaJ2018} that the set $\Quad_{\A}^c(n)$ is precisely the set $\{q_Q\}$ over all connected loop-less quivers $Q$ with $n$ arrows and $m=n-c+1$ vertices. This ``quiver realization'' facilitates the study of weak and strong Gram congruences within the set $\Quad_{\A}^c(n)$, as shown in~\cite{jaJ2020a}.
 \item[ii)] Using $(i)$, the definition of the cycle type $\ct(q)$ of a quadratic form $q$ in $\Quad_{\A}^c(n)$ follows from the notion of Coxeter-Laplace matrix $\Lambda_Q$ of a loop-less quiver $Q$ introduced in Theorem~\ref{T(O):main}, where it is shown that $\Lambda_Q$ is the permutation matrix of a permutation $\xi^-_Q$ of the set of vertices of $Q$. The construction of $\xi^-_Q$ and the proof of Theorem~\ref{T(O):main} is the purpose of Sections~\ref{(S):Min} and~\ref{(S):One}.
 \item[iii)] The strong Gram congruence invariance of the cycle type $\ct$ follows from Theorem~\ref{T(O):main} and some observations on the mapping $Q \mapsto q_Q$ presented in Lemma~\ref{L(O):well}.
 \item[vi)] Some technical considerations to determine the image of $\ct$ are given in Sections~\ref{(S):Two} and~\ref{(S):tHree}. In particular, in Definition~\ref{Ex(H)} we fix a set of quadratic forms $q$ in $\Quad_{\A}^c(n)$ representing those Coxeter polynomials permitted by Proposition~\ref{P(T):walk}.
 \item[v)] The Coxeter polynomial of a member of $\Quad_{\A}^c(n)$ is computed in Corollary~\ref{C:pol} with help of Theorem~\ref{T(O):main} (see also Algorithm~\ref{A:four}).
\end{itemize}

The Main Theorem is proved in Section~\ref{(S):Four}, collecting the results of previous sections. As application, in Corollary~\ref{C(F):num} we determine the (reduced) Coxeter number of any unit form $q$ in $\Quad_{\A}^c(n)$. In Section~\ref{(S):Five} we provide algorithms to compute the cycle type (Algorithms~\ref{A:one} and~\ref{A:two}) and Coxeter polynomials of such unit forms (Algorithms~\ref{A:three} and~\ref{A:four}), and comment on their spectral properties (Remark~\ref{RalgMult}).

All matrices in the paper have integer coefficients. The canonical basis of $\Z^n$ is denoted by $\bas_1,\ldots,\bas_n$ and the identity $n \times n$ matrix is denoted by $\Id_n$, and simply by $\Id$ for appropriate size. The transpose of a matrix $A$ is denoted by $A^{\tr}$, and if $A$ is an invertible square matrix, then $A^{-\tr}$ denotes $(A^{-1})^{\tr}$. If $A_1,\ldots,A_n$ are the columns of $A$, we write $A=[A_1 | A_2 |\ldots |A_n]$.

\section{Minimally monotonous walks} \label{(S):Min}

In this section we recall the definition of a quiver $Q$ and its (vertex-arrow) incidence matrix $I(Q)$. It was shown in~\cite[Proposition~4.4]{jaJ2020a} that if $Q$ has no loop, and $\widecheck{G}_Q$ is the upper triangular Gram matrix of $Q$ (defined below), then $I(Q)\widecheck{G}_Q^{-1}$ is also the incidence matrix of a loop-less quiver, called inverse quiver of $Q$ and denoted by $Q^{-1}$ (Proposition~\ref{INVERSE} below). The integer matrix $I(Q^{-1})=I(Q)\widecheck{G}_Q^{-1}$ plays a fundamental role in our discussion, and deserves a careful analysis. To this end, minimally monotonous (increasing or decreasing) walks were introduced in~\cite{jaJ2020a}. Here we use such walks to give an alternative description of $I(Q^{-1})$ (Lemma~\ref{L(M):Two}), which leads to part of the proof of Theorem~\ref{T(O):main}.

By \textbf{quiver} we mean a quadruple $Q=(Q_0,Q_1,\sou,\tar)$ such that $Q_0$ and $Q_1$ are finite sets (called \textbf{vertices} and \textbf{arrows} of $Q$ respectively), and $\sou,\tar:Q_1 \to Q_0$ are functions (called \textbf{source} and \textbf{target} function of $Q$). Since we want to associate to $Q$, unequivocally, an incidence matrix $I(Q)$, throughout the paper we assume that both the set of vertices and the set of arrows of any quiver are \textbf{totally ordered}, and write $i \leq j$ for arrows $i,j$ in $Q_1$, and $v \leq w$ for vertices $v,w$ in $Q_0$. We identify isomorphic quivers under the assumption that the isomorphism preserves the given orderings on the sets of vertices and arrows. Thus, if $|Q_1|=n$ and $|Q_0|=m$, we may assume without loss of generality that $Q_1=\{1,\ldots,n\}$ and $Q_0=\{1,\ldots,m\}$.

\medskip
The $m\times n$ \textbf{(vertex-arrow) incidence matrix} $I(Q)$ of $Q$ is given by,
\[
I(Q)=[I_{i_1} | \ldots |I_{i_n}], \quad \text{where $I_{i}=\bas_{\sou(i)}-\bas_{\tar(i)}$ for an arrow $i$},
\]
where $\bas_v$ is the $v$-th canonical vector in $\Z^m$ and $i_1,\ldots,i_n$ are the arrows in $Q$. Note that $I_i=0$ if and only if $i$ is a loop in $Q$. Observe also that if $Q'$ is a quiver obtained from $Q$ by a reordering of the set of arrows of $Q$, say via a permutation $\rho$ of $Q_1$, then $I(Q')=I(Q)P(\rho)$. Similarly, if $Q''$ is obtained from $Q$ by a reordering of the set of vertices of $Q$, say by a permutation $\xi$ of $Q_0$, then $I(Q'')=P(\xi)I(Q)$.

\medskip
For a quiver $Q=(Q_0,Q_1,\sou,\tar)$ and an arrow $i \in Q_0$ we take $\MiNu(i)=\{\sou(i),\tar(i)\}$, the set of vertices \textbf{incident} to arrow $i$. For a vertex $v$ and an arrow $i$ in $Q$, consider the following subsets of arrows of $Q$,
\begin{eqnarray}
 Q_1(v)& = & \{ j \in Q_1 \mid v \in \MiNu(j) \}, \nonumber \\
 Q_1^{<}(v,i)& = & \{ j \in Q_1 \mid v \in \MiNu(j) \text{ and } j < i \}, \nonumber \\
 Q_1^{\leq}(v,i)& = & \{ j \in Q_1 \mid v \in \MiNu(j) \text{ and } j \leq i \}, \nonumber
\end{eqnarray}
and take similarly $Q_1^{>}(v,i)$ and $Q_1^{\geq}(v,i)$.

\medskip
For a \textbf{walk} $\alpha=(v_{-1},i_0,v_{0},i_1,v_1,\ldots,v_{\ell-1},i_{\ell},v_{\ell})$ in $Q$ we use the notation
\[
\alpha=i_0^{\epsilon_0}i_1^{\epsilon_1}\cdots i_{\ell}^{\epsilon_{\ell}}, \quad \text{for signs $\epsilon_t=\pm 1$},
\]
where we take $\epsilon_t=+1$ if $\sou(i_t)=v_{t-1}$ and $\tar(i_t)=v_{t}$, and $\epsilon_t=-1$ in case $\sou(i_t)=v_{t}$ and $\tar(i_t)=v_{t-1}$ for $t=0,\ldots,\ell$ (as usual, exponents $+1$ are omitted). The integer $\ell+1$ is called \textbf{length} of $\alpha$, and if $\ell=-1$ then $\alpha$ is called a \textbf{trivial} walk. We take $\sou(\alpha)=v_{-1}$ and $\tar(\alpha)=v_{\ell}$, and call these vertices \textbf{origin} and \textbf{target} of the walks $\alpha$, respectively. The \textbf{reversed walk} of $\alpha$, denoted by $\alpha^{-1}$, is given by $\alpha^{-1}=i_{\ell}^{-\epsilon_{\ell}}i_{\ell-1}^{-\epsilon_{\ell-1}}\cdots i_{0}^{-\epsilon_{0}}$. The following special walks were considered in~\cite[Definition~4.1]{jaJ2020a}:
\begin{itemize}
 \item[a)] We say that the walk $\alpha=i_0^{\epsilon_0}i_1^{\epsilon_1}\cdots i_{\ell}^{\epsilon_{\ell}}$  is \textbf{minimally decreasing} if
\[
i_{t+1}=\max Q_1^{<}(v_t,i_t), \quad  \text{for $t=0,\ldots,\ell-1$}.
\]
 \item[b)] If $\alpha$ is minimally decreasing, we say that $\alpha$ is \textbf{left complete} if whenever $\beta \alpha$ is minimally decreasing for some walk $\beta$, then $\beta$ is a trivial walk. Similarly,  $\alpha$ is \textbf{right complete} if whenever $\alpha \beta$ is minimally decreasing for some walk $\beta$, then $\beta$ is a trivial walk. A left and right complete minimally decreasing walk will be called a \textbf{structural (decreasing) walk}.
\end{itemize}

We will mainly consider the following particular minimally decreasing walks. For an arrow $i$ there are exactly two right complete minimally decreasing walks starting with arrow $i$, one starting at vertex $\sou(i)$ and denoted by $\alpha^-_Q(i^{+1})$, and one starting at vertex $\tar(i)$ and denoted by $\alpha^-_Q(i^{-1})$. To be precise, if
\[
 \alpha_Q^-(i^{+1})=(v_{-1},i_0,v_{0},i_1,v_1,\ldots,v_{\ell-1},i_{\ell},v_{\ell}),
\]
then $v_{-1}=\sou(i)$, $i_0=i$, $i_{t+1}=\max Q_1^{<}(v_t,i_t)$ for $t=0,\ldots,\ell-1$,  and $Q_1^<(v_{\ell},i_{\ell})=\emptyset$. Similarly, if
\[
 \alpha_Q^-(i^{-1})=(v_{-1},i_0,v_{0},i_1,v_1,\ldots,v_{\ell-1},i_{\ell},v_{\ell}),
\]
then $v_{-1}=\tar(i)$, $i_0=i$, $i_{t+1}=\max Q_1^{<}(v_t,i_t)$ for $t=0,\ldots,\ell-1$,  and $Q_1^<(v_{\ell},i_{\ell})=\emptyset$. Note that the walks $\alpha_Q^-(i^{\pm 1})$ are determined by the initial vertex and the first arrow.

Consider now a vertex $v$ and take $i_0=\max Q_1(v)$. If $v=\sou(i_0)$ (resp. if $v=\tar(i_0)$) then $\alpha_Q^-(i_0^{+1})$ is also left complete (resp. $\alpha_Q^-(i_0^{-1})$ is also left complete), and it is therefore, a structural decreasing walk starting at $v$, denoted by $\alpha_Q^-(v)$. If $\gamma$ is an arbitrary structural decreasing walk starting at $v$, then the first arrow of $\gamma$ is necessarily $i_0$ (otherwise $\gamma$ could be extended on the left keeping the minimally decreasing property), and therefore, $\gamma=\alpha_Q^-(v)$.

\medskip
Dually, the \textbf{minimally increasing walks} $\alpha_Q^+(i^{\pm 1})$ are defined as follows. If
\[
 \alpha_Q^+(i^{+1})=(v_{-1},i_0,v_{0},i_1,v_1,\ldots,v_{\ell-1},i_{\ell},v_{\ell}),
\]
then $v_{-1}=\sou(i)$, $i_0=i$, $i_{t+1}=\min Q_1^{>}(v_t,i_t)$ for $t=0,\ldots,\ell-1$,  and $Q_1^>(v_{\ell},i_{\ell})=\emptyset$. Similarly, if
\[
 \alpha_Q^+(i^{-1})=(v_{-1},i_0,v_{0},i_1,v_1,\ldots,v_{\ell-1},i_{\ell},v_{\ell}),
\]
then $v_{-1}=\tar(i)$, $i_0=i$, $i_{t+1}=\min Q_1^{>}(v_t,i_t)$ for $t=0,\ldots,\ell-1$,  and $Q_1^>(v_{\ell},i_{\ell})=\emptyset$. If $i_0=\min Q_1(v)$, take $\alpha_Q^+(v):=\alpha_Q^+(i_0^{+1})$ if $v=\sou(i_0)$, and $\alpha_Q^+(v):=\alpha_Q^+(i_0^{-1})$ if $v=\tar(i_0)$. The following are straightforward observations.

\begin{remark}\label{L(M):One}
Let $Q$ be a loop-less quiver, with vertex $v \in Q_0$ and arrows $i,j \in Q_1$.
\begin{itemize}
 \item[i)] If $w=\tar(\alpha^-_Q(v))$, then $\alpha^+_Q(w)=\alpha^-_Q(v)^{-1}$.

 \item[ii)] If $i$ appears in $\alpha^+_Q(v)$ in the positive orientation (resp. in the negative orientation), then $\tar(\alpha^-_Q(i^{-1}))=v$ (resp. $\tar(\alpha^-_Q(i^{+1}))=v$).

 \item[iii)] If $v=\tar(\alpha^-_Q(i^{\pm 1}))$, then $\alpha^+_Q(v)=\alpha^-_Q(i^{\pm 1})^{-1}\gamma$, for some walk $\gamma$.
 \end{itemize}
\end{remark}

The \textbf{inverse quiver} of a loop-less quiver $Q=(Q_0,Q_1,\sou,\tar)$, as defined in~\cite[Definition~4.2]{jaJ2020a}, is the quiver $Q^{-1}=(Q_0^*,Q_1^*,\sou^*,\tar^*)$ with the same set of vertices as $Q$, the same number of arrows as $Q$ (that is, $Q_0^*=Q_0$ and $|Q_1^*|=|Q_1|$), and such that for each arrow $i$ in $Q$ there corresponds an arrow $i^*$ in $Q^{-1}$ with
\[
\sou^*(i^*)=\tar(\alpha^-_Q(i^{-1})), \quad \text{and} \quad \tar^*(i^*)=\tar(\alpha^-_Q(i^{+1})).
\]
The arrows in $Q^{-1}$ inherit the total ordering of the arrows in $Q$ via the correspondence $i \mapsto i^*$. When allowed by the context, we will drop the asterisk $*$ on arrows of $Q^{-1}$. The unique upper triangular matrix $\widecheck{G}_Q$ such that $I(Q)^{\tr}I(Q)=\widecheck{G}_Q+\widecheck{G}_Q^{\tr}$ is called the \textbf{triangular Gram matrix} of a quiver $Q$ (cf.~\cite[Definition 3.1]{jaJ2020a}).

\begin{proposition}[\cite{jaJ2020a}, Proposition~4.4]\label{INVERSE}
If $Q$ is a loop-less quiver, then $Q^{-1}$ is a loop-less quiver satisfying $(Q^{-1})^{-1}=Q$,
\[
I(Q^{-1})=I(Q)\widecheck{G}^{-1}_Q, \quad \text{and} \quad \widecheck{G}_{Q^{-1}}=\widecheck{G}^{-1}_Q.
\]
Moreover, $Q$ is connected if and only if $Q^{-1}$ is connected.
\end{proposition}

Recall from~\cite[\S 4.1]{jaJ2018} that if $\alpha=i_{0}^{\epsilon_{0}}i_{1}^{\epsilon_{1}}\cdots i_{\ell}^{\epsilon_{\ell}}$ is a walk in a quiver $Q$ with $n$ arrows, then the \textbf{incidence vector} of $\alpha$ is the vector $\vdim(\alpha) \in \Z^n$ given by
\[
\vdim(\alpha)=\sum_{t=0}^{\ell} \epsilon_t \bas_{i_t},
\]
where $\bas_i$ denotes the canonical vector of $\Z^n$ corresponding to arrow $i$. We will need the following alternative description of the inverse of a quiver.

\begin{lemma}\label{L(M):Two}
Let $Q$ be a loop-less quiver with vertices $v_1,\ldots, v_m$, and let $Q^{-1}$ be its inverse quiver. Then
\[
I(Q^{-1})^{\tr}=\left[ \vdim[\alpha_Q^+(v_1)] | \ldots | \vdim[\alpha_Q^+(v_m)] \right].
\]
\end{lemma}
\begin{proof}\vspace*{-2mm}
Let $b$ be the column of $I(Q^{-1})^{\tr}$ corresponding to vertex $v \in Q_0$. Hence
$b=\left[ \begin{matrix} b_1\\ \vdots \\ b_n \end{matrix}  \right]$, where\vspace*{-2mm}
\begin{equation*}
b_i = \left\{ \begin{array}{l l} +1, & \text{if $\sou^*(i^*)=v$},\\
-1, & \text{if $\tar^*(i^*)=v$},\\
0, & \text{if $v \notin \MiNu(i^*)$}.\end{array} \right.
\end{equation*}

On the other hand, take $\alpha_Q^+(v)=i_0^{\epsilon_0}i_1^{\epsilon_1}\cdots i_{\ell}^{\epsilon_{\ell}}=(v_{-1},i_0,v_{0},i_1,v_1,\ldots,v_{\ell-1},i_{\ell},v_{\ell})$ (hence $v_{-1}=v$). We prove that
\begin{itemize}
 \item[\emph{A)}] \textit{If $\epsilon_t=+1$ then $\sou^*(i_t^*)=v$, and if $\epsilon_t=-1$ then $\tar^*(i_t^*)=v$.}
 \item[\emph{B)}] \textit{If $v\in \MiNu(j^*)$ for some arrow $j^*$ in $Q^{-1}$, then $j=i_t$ for some $t \in \{0,\ldots,\ell \}$.}
\end{itemize}
\eject

To show claim $(A)$ assume first that $\epsilon_t=+1$, that is, $\sou(i_t)=v_{t-1}$ and $\tar(i_t)=v_t$. Then, by definition of $\sou^*$ and Remark~\ref{L(M):One}$(ii)$,
\[
\sou^*(i_t^*)=\tar(\alpha^-_Q(i_t^{-1}))=v_{-1}=v.
\]
Assume now that $\epsilon_t=-1$, that is, $\sou(i_t)=v_{t}$ and $\tar(i_t)=v_{t-1}$. Then, as before, we have $\tar^*(i_t^*)=\tar(\alpha^-_Q(i_t^{+1}))=v$.

\medskip
To show claim $(B)$ assume first that $\sou^*(j^*)=v$, that is,
\[
v=\tar(\alpha^-_Q(j^{-1})).
\]
By Remark~\ref{L(M):One}$(iii)$, there is a walk $\gamma$ such that $\alpha^+_Q(v)=\alpha^-_Q(j^{-1})^{-1}\gamma$. In particular, $j=i_t$ for some $t \in \{0,\ldots,\ell\}$. Assuming now that $\tar^*(j^*)=v$, then $v=\tar(\alpha^-_Q(j^{+1}))$, and we proceed analogously using Remark~\ref{L(M):One}$(iii)$.

\medskip
Finally, the identity $b=\vdim[\alpha^+_Q(v)]$ follows directly from $(A)$ and $(B)$.
\end{proof}

\section{Permutation of vertices determined by a quiver} \label{(S):One}

In this section we show the main technical result of the paper, Theorem~\ref{T(O):main}. The theorem introduces the Coxeter-Laplace matrix $\Lambda_Q$ of a loop-less quiver $Q$, and shows that it is a permutation matrix that can be obtained combinatorially from the structural walks of Section~\ref{(S):Min}. This construction yields one of the main definitions of the paper: the cycle type of a quiver.

\medskip
For a connected loop-less quiver $Q=(Q_0,Q_1)$, consider the function $\xi^-_Q:Q_0 \to Q_0$ given by,
\[
\xi^-_Q(v)=\tar(\alpha^-_Q(v)),
\]
and take similarly $\xi^+_Q(v)=\tar(\alpha^+_Q(v))$. Next we show that $\xi^-_Q$ is a permutation of $Q_0$, referred to as \textbf{permutation of vertices} associated to the quiver $Q$.

\begin{lemma} \label{R(M):inv}
For any loop-less quiver $Q$ and any vertex $v \in Q_0$ we have
\[
\xi^+_Q(\xi^-_Q(v))=v.
\]
In particular, $\xi^-_Q$ is invertible and $(\xi^-_Q)^{-1}=\xi^+_Q$.
\end{lemma}
\begin{proof}
Taking $w=\xi^-_Q(v)$, by Remark~\ref{L(M):One}$(i)$ we have $\alpha^+_Q(w)=\alpha^-_Q(v)^{-1}$, and
\[
\xi^+_Q(w)=\tar(\alpha^+_Q(w))=\tar(\alpha^-_Q(v)^{-1})=\sou(\alpha^-_Q(v))=v.
\]
In particular, $\xi^-_Q$ is injective, hence invertible with inverse $\xi^+_Q$.
\end{proof}

We need another preliminary observation.

\begin{remark} \label{R(M):inc}
For any loop-less quiver $Q$ with incidence matrix $I(Q)$, and any walk $\alpha$ in $Q$, we have
\[
I(Q)\vdim(\alpha)=\bas_{\sou(\alpha)}-\bas_{\tar(\alpha)}.
\]
\end{remark}
\begin{proof}
Clearly, the claim holds for trivial walks, and by definition of $I(Q)$ if $\alpha=i$ for some arrow $i$ (since $I(Q)=[I_1 |\ldots | I_n]$ where $I_i=\bas_{\sou(i)}-\bas_{\tar(i)} \in \Z^m$ and $\vdim(i)=\bas_i$). The claim holds similarly for $\alpha=i^{-1}$ (since $\vdim(i^{-1})=-\bas_{i}$).

\medskip
Now, for a concatenated walk $\alpha \beta$ we have $\vdim(\alpha \beta)=\vdim(\alpha)+\vdim(\beta)$, and therefore, by induction on the length of a walk,
\begin{eqnarray}
I(Q)\vdim(\alpha \beta) & = & I(Q)\vdim(\alpha)+I(Q)\vdim(\beta)=\bas_{\sou(\alpha)}-\bas_{\tar(\alpha)}+\bas_{\sou(\beta)}-\bas_{\tar(\beta)} \nonumber \\
& = & \bas_{\sou(\alpha)}-\bas_{\tar(\beta)}= \bas_{\sou(\alpha \beta)}-\bas_{\tar(\alpha \beta)}, \nonumber
\end{eqnarray}
since $\bas_{\tar(\alpha)}=\bas_{\sou(\beta)}$. This completes the proof.
\end{proof}

Let $Q$ be a loop-less quiver with incidence matrix $I(Q)$, inverse quiver $Q^{-1}$ and associated permutation of vertices $\xi^-_Q$. Denote by $\overline{Q}$ the underlying graph of $Q$. Let $\Inc(Q)$ be the \textbf{incidence bigraph} of $Q$ defined in~\cite[Definition~3.3]{jaJ2020a} (see also~\cite{jaJ2018}) as follows. The set of vertices $\Inc(Q)_0$ of $\Inc(Q)$ is the set of arrows of $Q$ (that is, $\Inc(Q)_0=Q_1$). The number of signed edges in $\Inc(Q)$ between vertices $i$ and $j$ is the cardinality of $\MiNu(i)\cap \MiNu(j)$. The sign of such arrows is $-1$ if $ij$ or $ji$ is a walk in $Q$, and it is $+1$ if $ij^{-1}$ or $i^{-1}j$ is a walk of $Q$. For a bigraph $\Delta$ denote by $\Adj(\Delta)$ the \textbf{upper triangular adjacency matrix} of $\Delta$ (resp. by $\SAdj(\Delta)=\Adj(\Delta)+\Adj(\Delta)^{\tr}$ the \textbf{symmetric adjacency matrix} of $\Delta$), and by $D_{\Delta}$  the diagonal matrix of degrees of $\Delta$.

\begin{theorem} \label{T(O):main}
Let $Q$ be a connected loop-less quiver with $m$ vertices and $n$ arrows. Then the following identities hold:
\begin{eqnarray}
G_Q:=I(Q)^{\tr}I(Q)&=&2\Id_n-\SAdj(\Inc(Q)), \nonumber \\
L_Q:=I(Q)I(Q)^{\tr}&=&D_{\overline{Q}}-\SAdj(\overline{Q}), \nonumber \\
\Cox_Q:=\Id_n-I(Q)^{\tr}I(Q^{-1})&=&-\widecheck{G}^{\tr}_Q\widecheck{G}^{-1}_Q,\nonumber \\
\LCox_Q:=\Id_m-I(Q^{-1})I(Q)^{\tr}&=&P(\xi^-_Q),\nonumber
\end{eqnarray}
where for a permutation $\rho$, the matrix $P(\rho)$ has as $i$-th column the canonical vector $\bas_{\rho(i)}$. Moreover,
\begin{itemize}
 \item[i)] The \textbf{Gram matrix} $G_Q$ of $Q$ has Dynkin type $\A_{m-1}$ and corank $n-m+1$, and every such Gram matrix can be obtained in this way.
 \item[ii)] The \textbf{Laplace matrix} $L_Q$ of $Q$ has corank one, with null space generated by the vector $\mathbbm{1}$ having all entries equal to $1$.
\end{itemize}
\end{theorem}

Since the matrix $\Cox_Q$ is the Coxeter-Gram matrix of the quadratic form $q_Q$, we refer to the matrix $\Lambda_Q$ as the \textbf{Coxeter-Laplace} matrix of $Q$.

\begin{proof}
The expression for $G_Q$ and claim $(i)$ were shown in~\cite{jaJ2018} (see also~\cite[Lemma~3.4 and Corollary~3.6]{jaJ2020a}).

\medskip
For a vertex $v \in Q_0$, consider the $v$-th column $b^v=I(Q)^{\tr}\bas_v$ of the matrix $I(Q)^{\tr}$. By definition, the entries of $b^v$ are indexed by the arrows of $Q$, and are given by
\begin{equation*}
b^v_i = \left\{ \begin{array}{l l} +1, & \text{if $\sou(i)=v$},\\
-1, & \text{if $\tar(i)=v$},\\
0, & \text{otherwise,}\end{array} \right.
\end{equation*}
since $Q$ has no loop. Then $(b^v)^{\tr}b^v=\sum_{i \in Q_1} (b^v_i)^2$ is precisely the number of arrows $i$ such that $\sou(i)=v$ or $\tar(i)=v$, that is, the degree of $v$ as a vertex in the underlying graph $\overline{Q}$ of $Q$. Moreover, for vertices $v \neq v'$ and an arrow $i \in Q_1$, we have $b^v_ib^{v'}_i=-1$ if $i$ joins vertices $v$ and $v'$ (in any direction), and $b^v_ib^{v'}_i=0$ otherwise. Then $-(b^v)^{\tr}b^{v'}=\sum_{i \in Q_1} -b^v_ib^{v'}_i$ is the number of edges in $\overline{Q}$ joining vertices $v$ and $v'$. Therefore, the identity
\[
I(Q)I(Q)^{\tr}=D_{\overline{Q}}-\SAdj(\overline{Q}),
\]
holds. To show $(ii)$, recall that the $i$-th row of $I(Q)^{\tr}$ is $\bas_{\sou(i)}-\bas_{\tar(i)}$, which implies that $I(Q)^{\tr}\mathbbm{1}=0$. Assume now that $Q$ is connected, and that $x \in \Z^m$ is a non-zero vector such that $I(Q)^{\tr}x=0$. Then $x_v=x_{v'}$ for any vertices $v$, $v'$ joint by an arrow in $Q$, thus, the connectivity of $Q$ implies that the vector $x$ is an integer multiple of $\mathbbm{1}$. This shows claim $(ii)$, since the null space of $L_Q=I(Q)I(Q)^{\tr}$ is the right null space of $I(Q)^{\tr}$.

\medskip
As shown in~\cite[Theorem~4.7]{jaJ2020a}, using Proposition~\ref{INVERSE} we have
\[
\Cox_Q=\Id_n-I(Q)^{\tr}I(Q^{-1})=\Id_n-I(Q)^{\tr}I(Q)\widecheck{G}_Q^{-1}=\Id_n-(\widecheck{G}_Q+\widecheck{G}_Q^{\tr})\widecheck{G}_Q^{-1}=-\widecheck{G}_Q^{\tr}\widecheck{G}_Q^{-1},
\]
since $I(Q)^{\tr}I(Q)=G_Q=\widecheck{G}_Q+\widecheck{G}_Q^{\tr}$. It remains to show that $\LCox_Q=P(\xi^-_Q)$.

\medskip
For any vertex $v \in Q_0$, Lemma~\ref{L(M):Two} and Remark~\ref{R(M):inc} yield
\begin{eqnarray}
\LCox_Q\bas_v & = & [\Id-I(Q^{-1})I(Q)^{\tr}]\bas_v=\bas_v-I(Q^{-1})\vdim[\alpha^+_{Q^{-1}}(v)] \nonumber \\
& = & \bas_{v}-[\bas_{\sou(\alpha^+_{Q^{-1}}(v))}-\bas_{\tar(\alpha^+_{Q^{-1}}(v))}] = \bas_{\xi^+_{Q^{-1}}(v)}, \nonumber
\end{eqnarray}
since $\sou(\alpha^+_{Q^{-1}}(v))=v$ and $\tar(\alpha^+_{Q^{-1}}(v))=\xi^+_{Q^{-1}}(v)$. This shows that $\LCox_Q=P(\xi_{Q^{-1}}^+)$.

\medskip
Using the identity $I(Q^{-1})=I(Q)\widecheck{G}_Q^{-1}$ from Proposition~\ref{INVERSE}, observe also that
\begin{eqnarray}
\LCox_Q\LCox_{Q^{-1}}& = & [\Id-I(Q^{-1})I(Q)^{\tr}][\Id-I(Q)I(Q^{-1})^{\tr}] \nonumber \\
& = & \Id-I(Q^{-1})I(Q)^{\tr}-I(Q)I(Q^{-1})^{\tr}+I(Q^{-1})I(Q)^{\tr}I(Q)I(Q^{-1})^{\tr} \nonumber \\
& = & \Id-I(Q^{-1})\widecheck{G}_Q^{\tr}I(Q^{-1})^{\tr}-I(Q^{-1})\widecheck{G}_QI(Q^{-1})^{\tr}+I(Q^{-1})G_QI(Q^{-1})^{\tr} \nonumber \\
& = & \Id-I(Q^{-1})\left[\widecheck{G}^{\tr}_Q+\widecheck{G}_Q-G_Q \right]I(Q^{-1})^{\tr}=\Id, \nonumber
\end{eqnarray}
that is, $(\xi^+_{Q})^{-1}=\xi^+_{Q^{-1}}$. By Lemma~\ref{R(M):inv} we get
\[
\LCox_Q=P(\xi^+_{Q^{-1}})=P((\xi^+_{Q})^{-1})=P(\xi^-_Q),
\]
which completes the proof.
\end{proof}

\section{Cycle type, restrictions and transpositions} \label{(S):Two}

This technical section presents some simple quiver constructions to obtain prescribed permutations of vertices. The main result, Proposition~\ref{P(T):walk}, describes the possible permutations obtained among connected loop-less quivers of a given number of vertices and arrows.

Let $\alpha$ be a walk in a loop-less quiver $Q$. Denote by $Q[\alpha]$ the quiver obtained from $Q$ by adding an arrow from $\sou(\alpha)$ to $\tar(\alpha)$, placed last in the total ordering in the arrows $Q_1$ of $Q$. For two vertices $v \neq w$ in $Q$, denote by $[v,w]$ the permutation of $Q_0$ that swaps vertices $v$ and $w$ (called \textbf{transposition} of $v$ and $w$).

\begin{lemma} \label{L(T):res}
Let $Q$ be a loop-less quiver, and take $Q'=Q^{(i)}$ the quiver obtained from $Q$ by removing the maximal arrow $i$ of $Q_1$. Then
\[
\xi^-_Q=\xi^-_{Q'}[\sou(i),\tar(i)].
\]
\end{lemma}
\begin{proof}
Let $v$ be a vertex of $Q$ with $v \notin \MiNu(i)$. Then $Q_1(v)=Q'_1(v)$ and $i > \max Q_1(v)$, by maximality of $i$. Hence $\alpha^-_Q(v)=\alpha^-_{Q'}(v)$, that is, $\xi^-_Q(v)=\xi^-_{Q'}(v)$. Observe that $\alpha^-_{Q}(\sou(i))=i\alpha^-_{Q'}(\tar(i))$ and that $\alpha^-_{Q}(\tar(i))=i^{-1}\alpha^-_{Q'}(\sou(i))$. Hence $\xi^-_{Q}(\sou(i))=\xi^-_{Q'}(\tar(i))$ and $\xi^-_{Q}(\tar(i))=\xi^-_{Q'}(\sou(i))$, which shows the claim.
\end{proof}

Let $\rho$ be a permutation of a finite set $Q_0$. The \textbf{cycle type} $\ct(\rho)$ of $\rho$ is the (non-increasing) sequence of cardinalities of the orbits of $\rho$. The cycle type $\ct(\rho)$ of $\rho$ is a partition of the integer $|Q_0|$.

\begin{definition} \label{D(O):ct}
For a connected loop-less quiver $Q$, define the \textbf{cycle type} $\ct(Q)$ of $Q$ as the cycle type of the permutation of vertices $\xi^-_Q$ determined by $Q$,
\[
\ct(Q):=\ct(\xi^-_Q).
\]
\end{definition}

By \textbf{Coxeter polynomial} of a loop-less quiver $Q$ we mean the characteristic polynomial $\cox_Q$ of the Coxeter matrix $\Cox_Q$ of $Q$. Recall that if $\chr_{M}(\va)$ denotes the characteristic polynomial of a square matrix $M$, and that if $A$ and $B$ are $m \times n$ and $n \times m$ matrices respectively, then
\[
\chr_{BA}(\va)=\va^{n-m}\chr_{AB}(\va),
\]
see for instance~\cite[\S 2.4]{fZ99}. As a consequence of Theorem~\ref{T(O):main}, we get the following particular description of corresponding Coxeter polynomials.

\begin{corollary} \label{C:pol}
Let $Q$ be a connected loop-less quiver. Then the Coxeter polynomial $\cox_Q$ of $Q$ is given by
\[
\cox_{Q}(\va)=(\va-1)^{c-1}\chr_{\ct(Q)}(\va),
\]
where $\chr_{\ct(Q)}(\va)=\prod \limits_{a=1}^{\ell(\pi)}(\va^{\pi_a}-1)$ if $\ct(Q)=(\pi_1,\ldots,\pi_{\ell(\pi)})$.
\end{corollary}
\begin{proof}
Using Theorem~\ref{T(O):main}, we have
\begin{eqnarray}
\cox_{Q}(\va)& = & \chr_{\Cox_Q}(\va)=\chr_{-I(Q)^{\tr}I(Q^{-1})}(\va-1) \nonumber \\
& = &(\va-1)^{n-m} \chr_{-I(Q^{-1})I(Q)^{\tr}}(\va-1)=(\va-1)^{n-m}\chr_{\LCox_Q}(\va) \nonumber \\
& = & (\va-1)^{c-1} \chr_{P(\xi^-_Q)}(\va), \nonumber
\end{eqnarray}
since $c-1=n-m$. The characteristic polynomial of permutation matrices is well known (cf.~\cite[\S 5.6]{fZ99}),
\[
\chr_{P(\xi^-_Q)}(\va)=\prod_{t=1}^{\ell(\pi)}(\va^{\pi_t}-1),
\]
where $\ct(\xi^-_Q)=\pi=(\pi_1,\ldots,\pi_{\ell(\pi)})$. This shows that $\cox_{Q}(\va)=(\va-1)^{c-1}\chr_{\ct(Q)}(\va)$, which completes the proof.
\end{proof}

\medskip
As alternative factorization of the Coxeter polynomial of $Q$, consider the polynomial $\nu_k(\va)=\va^{k-1}+\va^{k-2}+\ldots +\va+1$ for $k \geq 1$. Then $\va^k-1=(\va-1)\nu_k(\va)$, and
\begin{equation}\label{EQ:2}
\cox_Q(\va)=(\va-1)^{c+(\ell-1)}\prod_{t=1}^{\ell}\nu_{\pi_t}(\va),
\end{equation}
where $\ct(Q)=(\pi_1,\ldots,\pi_{\ell})$ and $(\ell-1)\geq 0$.

\begin{lemma}\label{L:tree}
Let $Q$ be a tree quiver. Then $\xi^-_Q$ is a cyclic permutation.
\end{lemma}
\begin{proof}
Consider the linear quiver $L$ with $m$ vertices,
\[
L=\xymatrix@C=1.5pc{v_1 \ar[r]^-{1} & v_2 \ar[r]^-{2} & v_3 \ar[r]^-{3} & v_4  \cdots  v_{m-2} \ar[r]^-{m-2} & v_{m-1} \ar[r]^-{m-1} & v_m.}
\]
Then $\xi^-_L(v_t)=v_{t+1}$ if $t<m$, and $\xi^-_L(v_m)=v_1$, that is, $\xi^-_L$ is a cyclic permutation. By equation~(\ref{EQ:2}), $\cox_L(\va)=\nu_m(\va)$.

\medskip
Assume now that $Q$ is an arbitrary tree quiver. Using~\cite[Corollary~3.11 and Proposition~3.13]{jaJ2020a}, we have $q_L \approx q_Q$, and in particular $\cox_Q(\va)=\cox_L(\va)$. If $\xi^-_Q$ is not a cyclic permutation, then $\ell=\ell(\ct(Q))>1$, and again by equation~(\ref{EQ:2}), the polynomial $\cox_Q$ has $1$ as a root. This is impossible since $1$ is not a root of $\cox_L(\va)=\nu_m(\va)$.
\end{proof}

Denote by $\Quiv_m(n)$ the set of connected loop-less quivers having $m$ vertices and $n$ arrows. We will also use the notation $\Quiv^c_m(n)$ where $c=n-m+1$, or simply $\Quiv^c(n)$.

\begin{proposition} \label{P(T):walk}
For a quiver $Q$ in $\Quiv_m(n)$, the cycle type $\ct(Q)$ of $Q$ is a partition in $\Part^c_1(m)$, where $c=n-m+1$.
\end{proposition}
\begin{proof}
That $\ct(Q) \vdash m$ is clear. We proceed by induction on $c \geq 0$. If $c=0$, then $Q$ is a tree, and by Lemma~\ref{L:tree}, $\xi^-_Q$ is a cyclic permutation. In particular $\ell(\ct(Q))=1$ and $\ct(Q)\in \Part^0_1(m)$.

\eject

Assume the claim holds for non-negative integers smaller than $c$. Fix $n>c$, take a quiver $Q \in \Quiv^c_m(n)$ and consider the quiver $Q^{(n)}$ obtained from $Q$ by removing the last arrow $n$.

\medskip
\noindent \textbf{Case 1.} Assume first that $Q^{(n)}$ is connected. Then $Q^{(n)} \in \Quiv^{c-1}_m(n-1)$, and by induction hypothesis we have $\xi^-_{Q^{(n)}} \in \Part_1^{(c-1)}(m)$, that is,
\[
0 \leq (c-1) - [\ell(\ct(Q^{(n)}))-1] \equiv 0 \mod 2.
\]
By Lemma~\ref{L(T):res} we have $\ell(\ct(Q))=\ell(\ct(Q^{(n)}))+\delta$, where $\delta=1$ if $\sou(n)$ and $\tar(n)$ belong to the same cycle of $\xi^-_Q$, and $\delta=-1$ otherwise. Hence
\begin{eqnarray}
0 &\leq & (c-1) - [\ell(\ct(Q^{(n)}))-1] \nonumber \\
&=&c-[\ell(\ct(Q^{(n)}))+\delta-1]+\delta-1 \nonumber \\
&=&c-[\ell(\ct(Q))-1]+\delta-1 \equiv 0 \mod 2. \nonumber
\end{eqnarray}
This shows that $\ct(Q) \in \Part^c_1(m)$, since $\delta-1 \leq 0$.

\medskip
\noindent \textbf{Case 2.} Assume now that $Q^{(n)}$ is not connected, that is,
$Q^{(n)}=Q^{\sou} \sqcup Q^{\tar}$ where $\sou(n) \in Q^{\sou}$ and $\tar(n) \in Q^{\tar}$. Note that $Q^{\sou} \in \Quiv^{c^{\sou}}_{m^{\sou}}(n^{\sou})$ and $Q^{\tar} \in \Quiv^{c^{\tar}}_{m^{\tar}}(n^{\tar})$ for non-negative integers $c^{\sou}$, $c^{\tar}$, $n^{\sou}$, $n^{\tar}$, $m^{\sou}$ and $m^{\tar}$ with $c^{\sou}+c^{\tar}=c$, $n^{\sou}+n^{\tar}=n-1$ and $m^{\sou}+m^{\tar}=m$. Thus, by induction on $c$, we may assume that
\[
\ct(\xi^-_{Q^{\sou}}) \in \Part^{c^{\sou}}_1(m^{\sou}), \quad \text{and} \quad \ct(\xi^-_{Q^{\tar}}) \in \Part^{c^{\tar}}_1(m^{\tar}),
\]
in case $c^{\sou},c^{\tar}>0$. If $c^{\sou}=0$ or $c^{\tar}=0$, we may use induction on $n$ to get the same conclusion, that is,
\[
0\leq c^{\sou}-[\ell(\ct(Q^{\sou}))-1] \equiv 0 \mod 2, \quad \text{and} \quad 0 \leq c^{\tar}-[\ell(\ct(Q^{\tar}))-1] \equiv 0 \mod 2.
\]
Note that, by Lemma~\ref{L(T):res}, we have $\ell(Q)=\ell(Q^{\sou})+\ell(Q^{\tar})-1$.  Therefore
\[
0 \leq (c^{\sou}-[\ell(\ct(Q^{\sou}))-1])+(c^{\tar}-[\ell(\ct(Q^{\tar}))-1])=c-[\ell(\ct(Q))-1] \equiv 0 \mod 2.
\]
We conclude that $\ct(Q) \in \Part^c_1(n)$.
\end{proof}

\section{Representative families of quivers} \label{(S):tHree}

In this section we fix connected non-negative unit forms of Dynkin type $\A_{r}$ having as Coxeter polynomial those permitted by Proposition~\ref{P(T):walk}. We need the following preliminary observation.

\begin{remark} \label{R(H):pair}
Let $Q$ be a loop-less quiver (not necessarily connected) with $n$ arrows. For any distinct vertices $v$ and $w$ in $Q$, let $Q'$ be the quiver obtained from $Q$ by adding a pair of parallel arrows from $v$ to $w$, labeled as
\[
\xymatrix@C=2.5pc{v \ar@<.5ex>[r]^-{n+1} \ar@<-.5ex>[r]_-{n+2} & w.}
\]
Then $\xi^-_{Q'}=\xi^-_Q$.
\end{remark}
\begin{proof}
Follows from Lemma~\ref{L(T):res}.
\end{proof}

Let $\pi=(\pi_1,\ldots,\pi_{\ell})$ be a partition of the integer $m \geq 2$, consisting of $\ell=\ell(\pi)$ parts. Observe that $\pi \in \Part_1^c(m)$ if and only if $c=\ell-1+2\dd$ for some integer $\dd \geq 0$. Below we determine a connected quiver $\overrightarrow{\A}^{\dd}[\pi]$ and its inverse $\overrightarrow{\Star}^{\dd}[\pi]$, with cycle type $\pi$ and corank $c$ for such $\pi$ and $\dd \geq 0$, (that is, connected quivers with $m$ vertices, $n=m+\ell+2(\dd-1)$ arrows and cycle type $\pi$). Roughly speaking, we start with a tree on $m$ vertices (which has cycle type~$(m)$ by Lemma~\ref{L:tree}), use Lemma~\ref{L(T):res} to break its cyclic components, and then apply Remark~\ref{R(H):pair} to obtain the correct corank, without modifying the associated cycle type.

\medskip
For a quiver $Q$ with vertices $v,w \in Q_0$, denote by $Q[v,w]$ the quiver obtained from $Q$ by adding an arrow from $v$ to $w$, placed last in the ordering of $Q_1$. Denote by $E_m=(\{v_1,\ldots,v_m\},\emptyset)$ the quiver with $m$ vertices $v_1,\ldots,v_m$ and no arrows, and consider the \textbf{linear quiver} $\overrightarrow{\A}_m$ and the \textbf{maximal star quiver} $\overrightarrow{\Star}_m$ each with $m-1$ arrows, given by
\[
\overrightarrow{\A}_m=E_m[v_1,v_2][v_2,v_3]\cdots [v_{m-2},v_{m-1}] [v_{m-1},v_m],
\]
and
\[
\overrightarrow{\Star}_m=E_m[v_1,v_2][v_1,v_3]\cdots [v_1,v_{m-1}] [v_1,v_m].
\]

\begin{definition}\label{Ex(H)}
For any partition $\pi=(\pi_1,\ldots,\pi_{\ell})$ of an integer $m \geq 2$, and any $\dd \geq 0$, consider the connected quivers $\overrightarrow{\A}^{\dd}[\pi]$ and $\overrightarrow{\Star}^{\dd}[\pi]$, with $m$ vertices and $n=m+\ell+2(\dd-1)$ arrows, defined as follows. If $\ell>1$, take the indices $i_1=m-\pi_1,i_2=m-(\pi_1+\pi_2),\ldots,i_{\ell-2}=m-(\pi_1+\ldots+\pi_{\ell-2})$, and $i_{\ell-1}=m-(\pi_1+\ldots+\pi_{\ell-1})=\pi_{\ell}$, (all of which belong to the set $\{1,\ldots,m-1\}$).
\begin{itemize}
 \item[i)] Take $\overrightarrow{\A}^{0}[\pi]=\overrightarrow{\A}_m[\tar(m-1),\sou(i_1)][\sou(i_1),\sou(i_2)]\cdots [\sou(i_{\ell-2}),\sou(i_{\ell-1})]$ if $\ell>1$, and $\overrightarrow{\A}^{0}[(m)]=\overrightarrow{\A}_m$ if $\ell=1$. Define recursively for $\dd>0$,
 \begin{equation*}
\overrightarrow{\A}^{\dd}[\pi]= \left\{ \begin{array}{l l}
\left(\overrightarrow{\A}^{\dd-1}[\pi]\right)[\sou(i_{\ell-1}),\sou(i_{\ell-2})][\sou(i_{\ell-2}),\sou(i_{\ell-1})], & \text{if $\ell>2$},\\
\left(\overrightarrow{\A}^{\dd-1}[\pi]\right)[\sou(i_1),\tar(m-1)][\tar(m-1),\sou(i_1)], & \text{if $\ell=2$},\\
\left(\overrightarrow{\A}^{\dd-1}[\pi]\right)[\tar(m-1),\sou(m-1)][\sou(m-1),\tar(m-1)], & \text{if $\ell=1$}.
\end{array} \right.
\end{equation*}

 \item[ii)] Take $\overrightarrow{\Star}^{0}[\pi]=\overrightarrow{\Star}_m[\sou(1),\tar(i_1)][\sou(1),\tar(i_2)]\cdots [\sou(1),\tar(i_{\ell-1})]$ if $\ell>1$, and $\overrightarrow{\Star}^{0}[(m)]=\overrightarrow{\Star}_m$ if $\ell=1$. Define recursively for $\dd>0$,
 \begin{equation*}
\overrightarrow{\Star}^{\dd}[\pi]= \left\{ \begin{array}{l l}
\left(\overrightarrow{\Star}^{\dd-1}[\pi]\right)[\sou(1),\tar(i_{\ell-1})][\sou(1),\tar(i_{\ell-1})], & \text{if $\ell>1$},\\
\left(\overrightarrow{\Star}^{\dd-1}[\pi]\right)[\sou(1),\tar(m-1)][\sou(1),\tar(m-1)], & \text{if $\ell=1$}.
\end{array} \right.
\end{equation*}
\end{itemize}
\end{definition}

For example, if $m=2$, $\pi=(1,1)$ and $\dd=1$, then $n=4$, $i_1=1$, and
\[
\overrightarrow{\A}^{1}[(1,1)]=\xymatrix@C=1.5pc{\bulito \ar@<3.3ex>[r]^-{1} \ar@<1.1ex>@{<-}[r]^-{2} \ar@<-1.1ex>[r]^-{3} \ar@<-3.3ex>@{<-}[r]^-{4} & \bulito} \qquad \qquad \overrightarrow{\Star}^{1}[(1,1)]=\xymatrix@C=1.5pc{\bulito \ar@<3.3ex>[r]^-{1} \ar@<1.1ex>[r]^-{2} \ar@<-1.1ex>[r]^-{3} \ar@<-3.3ex>[r]^-{4} & \bulito}
\]
If $m=7$, $\pi=(3,2,2)$ and $\dd=1$, then $n=10$, $i_1=4$, $i_2=2$, and we have
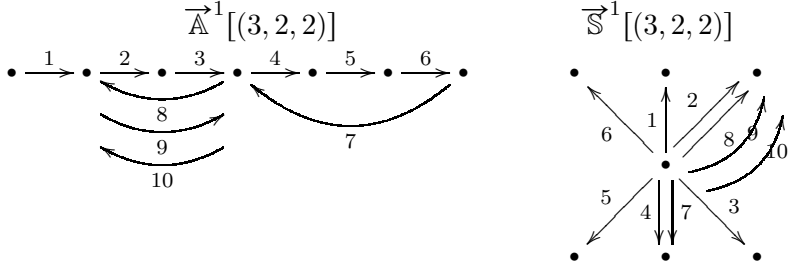
\[
\xy 0;/r.20pc/:
(-68, 5)="Fr1" *{};
( 68, 5)="Fr2" *{};
( 68,-40)="Fr3" *{};
(-68,-40)="Fr4" *{};
( -32, -5)="T1" *{\xymatrix@C=1.5pc{\bulito \ar[r]^-{1} & \bulito \ar[r]^-{2} & \bulito \ar[r]^-{3} & \bulito \ar[r]^-{4} \ar@/^10pt/[ll]^-{8} \ar@{<-}@<2.5ex>@/^10pt/[ll]^-{9} \ar@<5ex>@/^10pt/[ll]^-{10} & \bulito \ar[r]^-{5} & \bulito \ar[r]^-{6} & \bulito \ar@/^20pt/[lll]^-{7}}};
( 12, -5)="T2" *{\xymatrix{\bulito & \bulito & \bulito \\ & \bulito \ar[u]^-{1} \ar@<-.5ex>[ur]_-{8} \ar@/_10pt/@<-1ex>[ur]_-{9} \ar@/_10pt/@<-3ex>[ur]_-{10} \ar@<.5ex>[ur]^-{2} \ar[rd]^(.6){3} \ar@<-.5ex>[d]_-{4} \ar@<.5ex>[d]^-{7} \ar[ld]_-{5} \ar[lu]^-{6} & \\ \bulito & \bulito & \bulito}};
(-25,-2)="Fr3" *{\overrightarrow{\A}^{1}[(3,2,2)]};
( 37,-2)="Fr4" *{\overrightarrow{\Star}^{1}[(3,2,2)]};
\endxy
\]

\begin{remark}\label{R(H)minus}
For any partition $\pi$ of an integer $m \geq 2$, and any $\dd \geq 0$, the connected quivers  $\overrightarrow{\A}^{\dd}[\pi]$ and $\overrightarrow{\Star}^{\dd}[\pi]$ are loop-less and inverse of each other.
\end{remark}
\begin{proof}
Take $Q=\overrightarrow{\A}^{\dd}[\pi]=(Q_0,Q_1,\sou,\tar)$, and keep the notation of Definition~\ref{Ex(H)}. Observe first that if $\ell>1$, then
\[
\tar(m-1)=m>i_1>i_2>\ldots >i_{\ell-1}>0.
\]
Since the first $m-1$ arrows of $Q$ constitute the linear quiver $\overrightarrow{\A}_m$, then $Q$ is a connected loop-less quiver. Clearly, the same holds if $\ell=1$. Moreover, in any case we have $\tar(i)=\sou(i+1)$ for any $i=1,\ldots,n-1$. This shows that for any $i \in Q_1$,
\[
\alpha_Q^-(i^{-1})=i^{-1}(i-1)^{-1}\cdots 2^{-1}1^{-1},
\]
and therefore, $\sou^*(i^*)=\tar(1^{-1})=\sou(1)$ (see definition right before Proposition~\ref{INVERSE}). On the other hand,
 \begin{equation*}
\alpha_Q^-(i^{+1})= \left\{ \begin{array}{l l}
i, & \text{if $i=1,\ldots,m-1$},\\
i i_t, & \text{if $i=m,\ldots,m+\ell-2$}, \\
i \alpha_Q^-((i-1)^{+1}), & \text{if $i=m+\ell-1,\ldots,n$,}
\end{array} \right.
\end{equation*}
where the list $m,\ldots,m+\ell-2$ is empty if $\ell=1$. Then
\begin{equation*}
\tar^*(i^*)= \left\{ \begin{array}{l l}
\tar(i), & \text{if $i=1,\ldots,m-1$},\\
\tar(i_t), & \text{if $i=m,\ldots,m+\ell-2$}, \\
\tar(j), & \text{if $i=m+\ell-1,\ldots,n$,}
\end{array} \right.
\end{equation*}
where $j=i_{\ell-1}$ if $\ell>1$ and $j=m-1$ if $\ell=1$. Taking $Q'=\overrightarrow{\Star}^{\dd}[\pi]=(Q_0',Q_1',\sou',\tar')$, we observe directly form Definition~\ref{Ex(H)}$(ii)$ that $\sou^*=\sou'$ and $\tar^*=\tar'$, that is,
\[
\left( \overrightarrow{\A}^{\dd}[\pi] \right)^{-1}=\overrightarrow{\Star}^{\dd}[\pi].
\]
By Proposition~\ref{INVERSE}, the quiver $Q'$ is also connected and loop-less.
\end{proof}

Note that the column vector $\mathbbm{1}$ having all entries equal to $1$ is always a root of $q_{Q}$ for $Q=\overrightarrow{\A}^{\dd}[\pi]$, and that $q_{Q^{-1}}$ is always a weakly positive unit form (for $Q^{-1}=\overrightarrow{\Star}^{\dd}[\pi]$).

\section{Proof of main results} \label{(S):Four}

For a permutation $\rho$ of the vertices $Q_0$ of a quiver $Q=(Q_0,Q_1,\sou,\tar)$, denote by $\rho \cdot Q=(Q_0,Q_1',\sou',\tar')$ the quiver obtained by determining
\[
\sou'(i')=\rho(\sou(i)), \quad \text{and} \quad \tar(i')=\rho(\tar(i)),
\]
for an arrow $i'$ in $\rho \cdot Q$ corresponding to the arrow $i$ in $Q$. In other words, $\rho \cdot Q$ is the unique quiver satisfying $I(\rho \cdot Q)=P(\rho)I(Q)$. Observe that $\widecheck{G}_{\rho \cdot Q}=\widecheck{G}_{Q}$, and that $(\rho \cdot Q)^{-1}=\rho \cdot Q^{-1}$. Indeed,
\[
I((\rho \cdot Q)^{-1})=I(\rho \cdot Q)\widecheck{G}_{\rho \cdot Q}^{-1}=P(\rho)I(Q)\widecheck{G}_{Q}^{-1}=P(\rho)I(Q^{-1})=I(\rho \cdot Q^{-1}).
\]
The \textbf{quadratic form $q_Q$ associated to a quiver} $Q$ is given by $q_Q(x)=\frac{1}{2}||I(Q)x||^2$ for $x \in \Z^n$ (cf.~\cite[Definition~3.1]{jaJ2020a}).

\begin{lemma} \label{L(O):well}
Let $Q$ and $Q'$ be connected loop-less quivers with $n$ arrows and $m$ vertices.
\begin{itemize}
 \item[i)] If $q_{Q'}=q_{Q}$, then there is a permutation of vertices $\rho$ such that
 \[
\rho \cdot Q'=  Q, \quad \text{or} \quad \rho \cdot Q'=Q^{op},
 \]
 where $Q^{op}$ denotes the quiver obtained from $Q$ by changing the orientation of all arrows (the \textbf{opposite quiver} of $Q$).
 \item[ii)] If $q_{Q'} \sim q_Q$, then there is a permutation of vertices $\rho$ such that
 \[
 I(\rho \cdot Q')=I(Q)B,
 \]
 for some $\Z$-invertible matrix $B$.
 \item[iii)] If $q_{Q'} \approx q_{Q}$, then there is a permutation of vertices $\rho$ such that
\[
 I(\rho \cdot Q')=I(Q)B,  \quad \text{and} \quad \widecheck{G}_{Q'}=\widecheck{G}_{\rho \cdot Q'}=B^{\tr}\widecheck{G}_QB,
\]
for some $\Z$-invertible matrix $B$. In particular, $I(\rho \cdot (Q')^{-1})=I(Q^{-1})B^{-\tr}$.
\end{itemize}
\end{lemma}
\begin{proof}
For $(i)$, if $q_{Q'}=q_{Q}$, by~\cite[Corollary~7.3]{jaJ2018} there is either an isomorphism of quivers $(f_0,f_1):Q' \to Q$, or an isomorphism $(f_0,f_1):Q' \to Q^{op}$. This means that, taking $\rho=f_0$, we have $\rho \cdot Q'=  Q$ or $\rho \cdot Q'=Q^{op}$.

\medskip
To show $(ii)$, assume that there is a $\Z$-invertible matrix $C$ such that $q_{Q'}=q_QC$. Then the columns $c_1,\ldots,c_n$ of $C$ are roots of the unit form $q_Q$ (since $q_{Q'}$ is unitary), and by~\cite[Lemma~6.1]{jaJ2018} there are walks $\gamma_1,\ldots,\gamma_n$ in $Q$ such that $c_i=\vdim(\gamma_i)$ for $i=1,\ldots,n$. Denote by $Q''$ the quiver with $Q''_0=Q_0$ having an arrow $i \in Q''_1$ from $\sou(\gamma_i)$ to $\tar(\gamma_i)$ for each $i=1,\ldots,n$. Then, by Remark~\ref{R(M):inc}, we have $I(Q'')=I(Q)C$, and therefore
\[
q_{Q''}(x)=\frac{1}{2}x^{\tr}I(Q'')^{\tr}I(Q'')x=\frac{1}{2}x^{\tr}C^{\tr}I(Q)^{\tr}I(Q)Cx=q_Q(Cx)=q_{Q'}(x),
\]
for any $x \in \Z^n$. By $(i)$, there is a permutation $\rho$ of $Q_0$ with $\rho \cdot Q'=Q''$ or $\rho \cdot Q'=(Q'')^{op}$. Taking $B=C$ in the first case, and $B=-C$ in the second case, we get
\[
I(\rho \cdot Q')=I(Q'')=I(Q)C=I(Q)B, \quad \text{or} \quad I(\rho \cdot Q')=I((Q'')^{op})=-I(Q)C=I(Q)B,
\]
since $I(Q^{op})=-I(Q)$.

\medskip
To show $(iii)$, take $C$ such that $\widecheck{G}_{Q'}=C^{\tr}\widecheck{G}_{Q}C$.
By $(ii)$ we may assume that there is a permutation $\rho$ of $Q_0$ and a matrix $B$ such that $I(\rho \cdot Q')=I(Q)B$ and $\widecheck{G}_{Q'}=B^{\tr}\widecheck{G}_{Q}B$ (for $B=\pm C$). To show the last claim note that, using Proposition~\ref{INVERSE},
\[
I(\rho \cdot (Q')^{-1})=I(\rho \cdot Q')\widecheck{G}_{Q'}^{-1}=[I(Q)B][B^{\tr}\widecheck{G}_{Q}B]^{-1}=I(Q)\widecheck{G}_{Q}^{-1}B^{-\tr}=I(Q^{-1})B^{-\tr},
\]
which completes the proof.
\end{proof}

The following is our main definition. Denote by $\Part(m)$ the set of partitions of the integer $m \geq 2$.

\begin{definition}\label{DEF}
Take $0 \leq c<n$ and $m=n-c+1$. Assume that $q \in \Quad_{\A}^c(n)$, and that $Q \in \Quiv_m(n)$ is a loop-less quiver such that $q=q_Q$. We define a function $\ct:\Quad_{\A}^c(n) \to \Part(m)$ as the cycle type $\ct(q):=\ct(Q)$ of $Q$.
\end{definition}

By Lemma~\ref{L(O):well}$(i)$, the assignment $q \mapsto \ct(q)$ is well defined. Indeed, if $q=q_{Q'}$ for some other quiver $Q'$, then $\ct(Q')=\ct(\rho \cdot Q')=\ct((\rho \cdot Q')^{op})=\ct(Q)$.

\begin{theorem} \label{MAIN}
For any integers $0 \leq c < n$, the function $\ct$ given in Definition~\ref{DEF} is invariant under strong Gram congruence. Moreover, the image $\ct[\Quad_{\A}^c(n)]$ of $\ct$ is exactly $\Part^c_1(n-c+1)$, and for any $q$ in $\Quad_{\A}^c(n)$, the Coxeter polynomial of $q$ is given by
\[
\cox_q(\va)=(\va-1)^{c-1}\chr_{\ct(q)}(\va).
\]
\end{theorem}
\begin{proof}
Assume that $q' \approx q$, and choose quivers $Q$ and $Q'$ such that $q=q_Q$ and $q'=q_{Q'}$. By Lemma~\ref{L(O):well}$(iii)$, we may assume that there is a $\Z$-invertible matrix $B$ such that $I(Q')=I(Q)B$ and $I((Q')^{-1})=I(Q')B^{-\tr}$ (by replacing $\rho \cdot Q'$ by $Q'$ if necessary). Then
\[
\LCox_{Q'}=\Id-I((Q')^{-1})I(Q')^{\tr}=\Id-I(Q^{-1})B^{-\tr}B^{\tr}I(Q)^{\tr}=\Id-I(Q^{-1})I(Q)^{\tr}=\LCox_Q,
\]
and by Theorem~\ref{T(O):main}, we have $\xi^-_{Q'}=\xi^-_Q$. In particular,
\[
\ct(q')=\ct(Q')=\ct(\xi^-_{Q'})=\ct(\xi^-_Q)=\ct(Q)=\ct(q).
\]

Now, by definition and Proposition~\ref{P(T):walk}, the partition $\ct(q)$ belongs to the set $\Part^c_1(n-c+1)=\Part^c_1(m)$, for any quadratic form $q$ in $\Quad_{\A}^c(n)$. That any partition in $\Part^c_1(m)$ is the cycle type $\ct(q)$ of a quadratic form $q$ in $\Quad_{\A}^c(n)$ follows from Definition~\ref{Ex(H)}. Indeed, take $\pi \in \Part^c_1(m)$ with $\ell=\ell(\pi)$, and consider the quiver $Q=\overrightarrow{\A}^{\dd}[\pi]$ where $c=\ell-1+2\dd$. Then $q_Q \in \Quad^c_{\A}(n)$ and $\ct(q_Q)=\pi$. The description of Coxeter polynomials was shown in Corollary~\ref{C:pol}.
\end{proof}

Recall that the Coxeter matrix $\Cox_q$ of a non-negative unit form $q$ is a weakly periodic matrix, that is, there is a minimal $k \geq 1$ such that $\Id-\Cox_q^k$ is a nilpotent matrix (cf.~\cite{mS05} or~\cite{BJP19}). Such minimal power, denoted by $\RcoxN(q)=k$, is called \textbf{reduced Coxeter number} of $q$. In case $\Cox_q^k=\Id$ for some minimal $k \geq 1$, then $\coxN(q)=k$ is called \textbf{Coxeter number} of $q$, otherwise we set $\coxN(q)=\infty$.

\begin{corollary}\label{C(F):num}
Let $q$ be a unit form in $\Quad_{\A}^c(n)$, and consider its cycle type  $\ct(q)=(\pi_1,\ldots,\pi_{\ell})$. Then
\begin{itemize}
 \item[i)] The Coxeter number $\coxN(q)$ of $q$ is finite if and only if $\ell=1$, in which case $\coxN(q)=\pi_1$.
 \item[ii)] The reduced Coxeter number $\RcoxN(q)$ of $q$ is given by
\[
\RcoxN(q)=\lcm (\pi_1,\ldots,\pi_{\ell}),
\]
where $\lcm$ denotes \emph{least common multiple}.
\end{itemize}
\end{corollary}
\begin{proof}
Take $q=q_Q$ for some quiver $Q$. By Theorem~\ref{T(O):main} we have $\Cox_q=\Cox_Q$. Let us first show that $\Cox_Q^k=\Id-I(Q)^{\tr}\nu_k(\LCox_Q)I(Q^{-1})$ for any $k \geq 1$, where $\nu_k(\va)$ is the polynomial $\nu_k(\va)=\va^{k-1}+\va^{k-2}+\ldots+\va+1$. Indeed, by induction on $k$ and from Theorem~\ref{T(O):main}, we have
\begin{eqnarray}
 \Cox_Q^k & = & \Cox_Q\Cox_Q^{k-1}=[\Id-I(Q)^{\tr}I(Q^{-1})][\Id-I(Q)^{\tr}\nu_{k-1}(\LCox_Q)I(Q^{-1})] \nonumber \\
& = & \Id-I(Q)^{\tr}[\Id+\nu_{k-1}(\LCox_Q)-I(Q^{-1})I(Q)^{\tr}\nu_{k-1}(\LCox_Q)]I(Q^{-1}) \nonumber \\
& = & \Id-I(Q)^{\tr}[\Id+[\Id-I(Q^{-1})I(Q)^{\tr}]\nu_{k-1}(\LCox_Q)]I(Q^{-1}) \nonumber \\
& = & \Id-I(Q)^{\tr}[\Id+\LCox_Q\nu_{k-1}(\LCox_Q)]I(Q^{-1}) \nonumber \\
& = & \Id-I(Q)^{\tr}\nu_k(\LCox_Q)I(Q^{-1}). \label{EQOne}
\end{eqnarray}

For a vertex $v \in Q_0$ and an integer $a \geq 0$, take $v_a=(\xi^-_Q)^a(v)$. Let $\beta$ be a walk in $Q^{-1}$ from a vertex $v$ to a  different vertex $w$. Note that, using equation $\Lambda_Q=P(\xi^-_Q)$ of Theorem~\ref{T(O):main}, and Remark~\ref{R(M):inc}, we have
\begin{eqnarray}
I(Q)[\Id-\Cox_Q^k]\vdim(\beta)&=&I(Q)I(Q)^{\tr}\nu_k(\LCox_Q)I(Q^{-1})\vdim(\beta) \nonumber \\ & = & L_Q\nu_k(\LCox_Q)(\bas_v-\bas_w) \nonumber \\
& = & L_Q[(\bas_{v_0}+\ldots + \bas_{v_{k-1}})-(\bas_{w_0}+\ldots+\bas_{w_{k-1}})]. \label{EQTwo}
\end{eqnarray}
Recall that, since $Q$ is connected, the null space of the Laplace matrix $L_Q$ is generated by the (column) vector $\mathbbm{1} \in \Z^m$ with all entries equal to $1$ (Theorem~\ref{T(O):main}$(ii)$).

\medskip
To show $(i)$, assume first that $\ell=1$. Then $\nu_{\pi_1}(\Lambda_Q)=\pi_1 [\mathbbm{1}\mathbbm{1}^{\tr}]$ (that is, the matrix with all entries equal to $\pi_1$), and therefore, $\Cox_Q^{\pi_1}=\Id$ by equation~(\ref{EQOne}). Now, if $1\leq k<\pi_1$, taking a walk from $v$ to $w=v_1$, from~(\ref{EQTwo}) we get
\[
I(Q)[\Id-\Cox_Q^k]\vdim(\beta)=L_Q(\bas_{v_0}-\bas_{v_k}) \neq 0,
\]
since $v_0 \neq v_k$ (for $\ell=1$ and $k<\pi_1$). This shows that the Coxeter number of $q$ is $\coxN(q)=\pi_1$.

Assume now that $\ell>1$. Note that if $v$ and $w$ belong to different $\xi^-_Q$ orbits in $Q_0$, then
\[
(\bas_{v_0}+\ldots + \bas_{v_{k-1}})-(\bas_{w_0}+\ldots+\bas_{w_{k-1}}) \notin \Z \mathbbm{1},
\]
and therefore, by~(\ref{EQTwo}) we have $\Cox_Q^k \neq \Id$ for any $k \geq 1$, which completes the proof of $(i)$.

\medskip
Observe now that,
\begin{eqnarray}
[\Id- \Cox_Q^k][\Id-\Cox_Q] & = & [I(Q)^{\tr}\nu_{k}(\LCox_Q)I(Q^{-1})][I(Q)^{\tr}I(Q^{-1})] \nonumber \\
& = & \nonumber I(Q)^{\tr}[\nu_k(\LCox_Q)(\Id-\LCox_Q)]I(Q^{-1})  \\
& = & I(Q)^{\tr}[\Id-\LCox^k_Q]I(Q^{-1}),\label{EQThree}
\end{eqnarray}
and
\begin{eqnarray}
[\Id- \Cox_Q^k]^2 & = & [I(Q)^{\tr}\nu_{k}(\LCox_Q)I(Q^{-1})][I(Q)^{\tr}\nu_k(\LCox_Q)I(Q^{-1})] \nonumber \\
& = & \nonumber I(Q)^{\tr}[\nu_k(\LCox_Q)(\Id-\LCox_Q)\nu_k(\LCox_Q)]I(Q^{-1})  \\
& = & I(Q)^{\tr}[\nu_k(\LCox_Q)(\Id-\LCox^k_Q)]I(Q^{-1}).\label{EQFour}
\end{eqnarray}

To show $(ii)$, recall that the order of $\xi^-_Q$ is $\lcm(\ct(\xi^-_Q))$. Since $\LCox_Q=P(\xi^-_Q)$, by~(\ref{EQThree}) we have $[\Id- \Cox_Q^k][\Id-\Cox_Q]=0$ if $k=\lcm(\ct(q))$, that is,
\[
[\Id- \Cox_Q^k]^2=[\Id- \Cox_Q^k][\Id- \Cox_Q]\nu_k(\Cox_Q)=0.
\]
Assume now that $k<\lcm(\ct(q))$, and choose a vertex $v \in Q_0$ such that $v \neq v_k$. Take $w=v_1$ and $\beta$ a walk from $v$ to $w$ in $Q^{-1}$. Similarly, as in~(\ref{EQTwo}), by~(\ref{EQFour}) we have
\begin{eqnarray}
I(Q)[\Id- \Cox_Q^k]^2\vdim(\beta) & = & L_Q\nu_k(\LCox_Q)(\Id-\LCox_Q^k)(\bas_{v_0}-\bas_{v_1}) \nonumber \\
& = & L_Q\nu_k(\LCox_Q)[(\bas_{v_0}-\bas_{v_1})-(\bas_{v_k}-\bas_{v_{k+1}})] \nonumber \\
& = & L_Q[\bas_{v_0}+\bas_{v_{2k}}-2\bas_{v_k}] \neq 0, \nonumber
\end{eqnarray}
since $\bas_{v_0}+\bas_{v_{2k}}-2\bas_{v_k} \notin \Z\mathbbm{1}$, for $v_0 \neq v_k$. This shows that $[\Id- \Cox_Q^k]^2 \neq 0$ for $k<\lcm(\ct(q))$, which completes the proof.
\end{proof}

To illustrate the main results, we end this section with some examples of Coxeter polynomials and (reduced) Coxeter numbers that occur among quadratic forms $q \in \Quad^c_{\A}(n)$ for small $n$ and $c$. For instance, if $n=5$ and $c=2$, then $m=n-c+1=4$ and the set
\[
\Part_1^2(4)=\{\pi \vdash m \mid \ell(\pi) \in \{1,3\}\},
\]
contains only two partitions, namely $(4)$ and $(2,1,1)$. By Theorem~\ref{MAIN} and Corollary~\ref{C(F):num}, for $q \in \Quad^2_{\A}(5)$ we have either
\[
\cox_q(\va)=(\va^4-1)(\va-1), \quad \text{or} \quad \cox_q(\va)=(\va^2-1)(\va-1)^3,
\]
with corresponding Coxeter numbers $4$ and $\infty$, and reduced Coxeter numbers $4$ and $2$. Similarly, if $n=8$ and $c=4$, then $m=n-c+1=5$, and
\[
\Part_1^4(5)=\{\pi \vdash m \mid \ell(\pi) \in \{1,3,5\}\}.
\]
 The four partitions in $\Part_1^4(5)$, and corresponding Coxeter polynomials and (reduced) Coxeter numbers among the unit forms in $\Quad_{\A}^4(8)$, are listed in the following table.

\eject

\begin{center}
\begin{tabular}{l c c c}
Partition & Coxeter polynomial & Coxeter number & Reduced Coxeter number \\
\hline \\
(5) & $(\va^5-1)(\va-1)^3$ & 5 & 5 \\
(3,1,1) & $(\va^3-1)(\va-1)^5$ & $\infty$ & 3 \\
(2,2,1) & $(\va^2-1)^2(\va-1)^4$ & $\infty$ & 2 \\
(1,1,1,1,1) & $(\va-1)^8$ & $\infty$ & 1
\end{tabular}
\end{center}

\section{Comments and algorithms} \label{(S):Five}

An important problem in the theory of quadratic forms, or the corresponding graphical formulation in terms of edge-bipartite graphs developed by Simson and collaborators~\cite{dS16a,dS20,SZ17}, is to find characterizations for the strong Gram congruence. So far, the pair consisting of the Dynkin type and the spectrum of the Coxeter polynomial of a connected non-negative unit form $q$ (the so-called \emph{Coxeter-Dynkin type} of $q$), seems to be a good candidate for such characterization (see for instance~\cite[Problem~1.9]{dS16a}).

\subsection*{Problem A.} 
Let $q$ and $\widetilde{q}$ be connected weakly Gram congruent non-negative unit forms. If the Coxeter polynomials of $q$ and $\widetilde{q}$ coincide, are the unit forms $q$ and $\widetilde{q}$ strongly Gram congruent?

\medskip
An affirmative answer to Problem~A for positive unit forms with small number of variables (not exceeding $9$), including the exceptional cases $\E_6$, $\E_7$ and $\E_8$, is part of the work of Simson~\cite{dS20} and collaborators~\cite{BFS14,KS15a,KS15b,KS15c}, aiming much general problems on edge-bipartite graphs, morsifications and mesh-geometries. The positive cases of Dynkin type $\D_n$ and $\A_n$ were solved recently in~\cite{dS21a} and~\cite{dS21b} respectively, also in a wider context, and similar results for principal unit forms associated to posets were shown in~\cite{GSZ14}. The main construction of the paper, the cycle type, allows a reformulation of Problem~A for the case of Dynkin type $\A_r$:

\subsection*{Problem B.} 
Does the cycle type assignment induce a bijection
\[
\ct:[\Quad_{\A}^c(n)/ \approx] \longrightarrow \Part^c_1(n-c+1),
\]
for any $n \geq 1$ and $0 \leq c <n$?

\medskip
The bijectivity of $\ct$ for the cases $c \in \{0,1\}$ is consequence of the main results Theorems~3.16 and~4.12 of~\cite{jaJ2020a}, concerning the positive and principal cases of Dynkin type $\A_n$ respectively. In an upcoming work~\cite{jaJ2020c} we will approach Problem~B in full generality with matricial techniques.

Let $\CSpec(q)$ denote the \textbf{Coxeter spectrum} of a connected non-negative unit form $q$, that is, the multi-set of roots of the Coxeter polynomial $\cox_q(\va)$. By non-negativity, every $\va \in \CSpec(q)$ is a root of unity (cf.~\cite[Theorems~2.6 and~3.4]{mS05}).

\begin{remark}\label{RalgMult}
Let $q$ be a unit form in $\Quad_{\A}^c(n)$, for $n \geq 1$ and $0 \leq c<n$, and consider the cycle type $\ct(q)=(\pi_1,\ldots,\pi_{\ell})$ of $q$.
\begin{itemize}
 \item[a)] If $\eta$ is a primitive $d$-root of unity for some $d > 1$, then the multiplicity of $\eta$ in $\CSpec(q)$ is
 \[
 \#\{ \text{$a \in \{1,\ldots,\ell\}$ such that $d$ divides $\pi_a$}\},
 \]
where $\# S$ denotes the cardinality of a set $S$.
 \item[b)] The multiplicity of $1$ in $\CSpec(q)$ is $c+(\ell-1)$.
\end{itemize}
\end{remark}
\begin{proof}
Recall from Corollary~\ref{C:pol} that the Coxeter polynomial of $q$ is given by
\[
\cox_q(\va)=(\va-1)^{c-1}\prod_{a=1}^{\ell}(\va^{\pi_a}-1).
\]
Let $\eta$ be a primitive $d$-root of unity for some $d > 1$. It is well known that $\eta$ is a root of $(\va^t-1)$ if and only if $d$ divides $t$, and in that case the multiplicity of $\eta$ is one (see for instance~\cite[\S 3.3]{Pra01}). This shows claim $(a)$. To show $(b)$ consider the alternative factorization~(\ref{EQ:2}) of $\cox_q(\va)$ given right after Corollary~\ref{C:pol},
\[
\cox_Q(\va)=(\va-1)^{c+(\ell-1)}\prod_{t=1}^{\ell}\nu_{\pi_t}(\va).
\]
Thus, claim $(b)$ holds since $\nu_t(1)\neq 0$ for any $t \geq 0$.
\end{proof}

With Problem~B in mind, the unit forms associated to the representative quivers of Section~\ref{(S):tHree} are proposed representatives of the strong Gram classes in $\Quad_{\A}^c(n)$, playing the role of the canonical extensions of $\A_r$ defined by Simson in~\cite{dS16a} for the weak Gram congruence. The following straightforward observation, relating Simson's construction with those in Section~\ref{(S):tHree}, will be useful for our computations.

\begin{remark}\label{R(H)}
For $r \geq 1$ and $c \geq 0$, consider the quiver with $r+1$ vertices and $r+c$ arrows, given by
 \begin{equation*}
Q^c_r =  \left\{ \begin{array}{l l}
\overrightarrow{\A}^{\frac{c}{2}}[(r+1)]\mathcal{V}^c_r\mathcal{T}, & \text{if $c$ is even},\\
\overrightarrow{\A}^{\frac{c-1}{2}}[(r,1)]\mathcal{V}^c_r, & \text{if $c$ is odd},
\end{array} \right.
\end{equation*}
where, if $c$ is even, $\mathcal{V}^c_r$ is the inversion of the arrows $r+2i$ for $i=1,\ldots,\frac{c}{2}$, and $\mathcal{T}$ is the iterated flation (see definition and notation in~\cite[\S 2.5]{jaJ2020a}) given by
\[
\mathcal{T}=\mathcal{T}^{1}\cdots \mathcal{T}^{c/2}, \quad \text{where $\mathcal{T}^{i}:=\mathcal{T}^{+1}_{r+i,r-1}\cdots \mathcal{T}^{+1}_{r+i,2}\mathcal{T}^{+1}_{r+i,1}$, for $i=1,\ldots,c/2$}.
\]
If $c$ is odd,  $\mathcal{V}^c_r$ is the inversion of the arrows $r+2i$ for $i=1,\ldots,\frac{c-1}{2}$. To be precise,
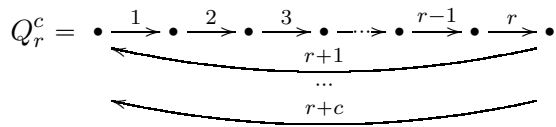
\[
Q^c_r=\xymatrix@C=1.5pc{\bulito \ar[r]^-{1} & \bulito \ar[r]^-{2} & \bulito \ar[r]^-{3} & \bulito \ar[r]|{\cdots} & \bulito \ar[r]^-{r-1} & \bulito \ar[r]^-{r} & \bulito \ar@<1ex>@/^10pt/[llllll]_-{r+1}^-{\cdots} \ar@<5ex>@/^10pt/[llllll]_-{r+c} }
\]
Let $\widehat{\A}^{(c)}_r$ be the canonical $c$-extension of $\A_r$ defined in~\cite{dS16a}. Then $\widehat{\A}^{(c)}_r = \Inc\left( Q^c_r \right)$.
\end{remark}

Next we show how to find a quiver $Q$ with $n$ arrows such that $q=q_Q$, given that $q$ is a connected non-negative unit form of Dynkin type $\A_{r}$ in $n \geq r$ variables, following~\cite[Proposition~3.15]{jaJ2020a}.

\begin{algorithm}\label{A:one}\small $ $\par

\smallskip
\textbf{Input:} A connected non-negative quadratic unit form $q$ in $n \geq 1$ variables, and of Dynkin type $\A_r$ for some $r \geq 1$.

\smallskip
\textbf{Output:} A connected loop-less quiver $Q$ with $n$ arrows and $m=r+1$ vertices, such that $q=q_Q$.

\smallskip
\underline{Step 1.} Compute the upper triangular Gram matrix $\widecheck{G}_q$ of $q$, and the corresponding symmetric Gram matrix $G_q=\widecheck{G}_q+\widecheck{G}_q^{\tr}$. Recall that $\bas_i G_q \bas_j=q(\bas_i+\bas_j)-q(\bas_i)-q(\bas_j)$, for any canonical vectors $\bas_i, \bas_j$ in $\Z^n$.

\smallskip
\underline{Step 2.} Find a $\Z$ invertible matrix $B$ such that $G_{\widehat{\A}^{(c)}_r}=B^{\tr}G_qB$, where $G_{\widehat{\A}^{(c)}_r}$ denotes the symmetric Gram matrix of the canonical $c$-extension $\widehat{\A}^{(c)}_r$ of $\A_r$. For instance, use Algorithm~3.18 in~\cite{SZ17}.

\smallskip
\underline{Step 3.} Calculate the inverse matrix $B^{-1}$, and take $I:=I(Q^c_r)B^{-1}$ where $Q^c_r$ is the quiver given in Remark~\ref{R(H)}. Verify that
\[
I^{\tr}I=B^{-\tr}I(Q^c_r)^{\tr}I(Q^c_r)B^{-1}=B^{-\tr}G_{\widehat{\A}^{(c)}_r}B^{-1}=G_q.
\]

\underline{Step 4.} Take $Q_0=\{1,\ldots,m=r+1\}$ and $Q_1=\{1,\ldots,n\}$. For every $i \in Q_1$, the column vector $b=I\bas_i$ satisfies $b^{\tr}b=\bas_i^{\tr}G_q\bas_i=2$, since $q$ is unitary. Then there are different indices $s,t \in Q_0$ with
\[
b=S\bas_s+T\bas_t, \quad \text{for some signs $S,T \in \{\pm 1\}$}.
\]
Moreover, $\mathbbm{1}^{\tr}I=\mathbbm{1}^{\tr}I(Q^c_r)B^{-1}=0$, which implies that $\mathbbm{1}^{\tr}b=0$. Then, after switching the labels $s$ and $t$ if necessary, we may assume that $S=+1$ and $T=-1$. Take $\sou(i)=s$ and $\tar(i)=t$, which defines a quiver $Q=(Q_0,Q_1,\sou,\tar)$ with $I=I(Q)$. That $Q$ has no loop is clear, since $s \neq t$ for an arrow $i$ as above. That $Q$ is connected follows from~\cite[Lemma~3.4$(d)$]{jaJ2020a}. By Step~3 and the definition $q_Q(x):=\frac{1}{2}||I(Q)x||^2$ for $x \in \Z^n$, we have
\[
q(x)=\frac{1}{2}x^{\tr}G_qx=\frac{1}{2}x^{\tr}I(Q)^{\tr}I(Q)x=q_Q(x),
\]
as wanted.
\end{algorithm}\normalsize

The cycle type $\ct(q)$ of a quadratic form $q$ in $\Quad_{\A}^c(n)$, for any $n \geq 1$ and any $ 0 \leq c < n$, can be found directly from a quiver $Q$ with $q=q_Q$. Indeed, compute first the permutation $\xi^-_Q$ (either by constructing the structural decreasing walks $\alpha^-_Q(v)$ for any vertex $v$, or directly by computing the matrix $\Lambda_Q=\Id_m-I(Q)\widecheck{G}_q^{-1}I(Q)^{\tr}=P(\xi^-_Q)$, cf.~Theorem~\ref{T(O):main} and Proposition~\ref{INVERSE}). Then find a cycle decomposition of $\xi^-_Q$, using for instance the \texttt{full\_cyclic\_form()} sympy Python library function, and store the corresponding lengths, ordered non-increasingly, in a list $\ct(q)$. We stress that the cycle type $\ct(q)$ can be recovered from the Coxeter polynomial of $q$ (or its spectrum), as indicated in Algorithm~\ref{A:two} below. We need the following straightforward observation.

\begin{remark}\label{A:r}
Let $\pi=(\pi_1,\ldots,\pi_{\ell}) \vdash m$ be a partition of the integer $m \geq 1$. Then
\[
\pi_1=\max \{ \text{$t \geq 1$ such that $(\va^t-1)$ divides $\chr_{\pi}(\va)$}\}.
\]
\end{remark}
\begin{proof}
Take $m_0:=\max \{ \text{$t \geq 1$ such that $(\va^t-1)$ divides $\chr_{\pi}(\va)$}\}$. Since $\chr_{\pi}(\va)=\prod_{a=1}^{\ell}(\va^{\pi_a}-1)$, clearly $m_0 \geq \max\{\pi_1,\ldots,\pi_{\ell}\}= \pi_1$. On the other hand, since $(\va^{m_0}-1)$ divides $\chr_{\pi}(\va)$, any primitive $m_0$-root of unity is a root of $\chr_{\pi}(\va)$. By Remark~\ref{RalgMult}$(a)$, this means that $m_0$ divides $\pi_a$ for some $a \in \{1\ldots,\ell\}$. In particular $m_0 \leq \pi_1$, and the claim follows.
\end{proof}

\begin{algorithm}\small \label{A:two}$ $\par

\smallskip
\textbf{Input:} The Coxeter polynomial $\cox_q(\va)$ of a quadratic form $q$ in $\Quad_{\A}^c(n)$, for $n \geq 1$ and $0 \leq c < n$.

 \smallskip
\textbf{Output:} The cycle type $\ct(q)$ of $q$.

\smallskip
\underline{Step 1.} Take an empty list $\ct(q)=\emptyset$.

\smallskip
\underline{Step 2.} By Theorem~\ref{MAIN}, there is a polynomial $p_0(\va)$ such that $\cox_q(\va)=(\va^{c-1}-1)p_0(\va)$, and a partition $\pi^0$ of $m_0=m$ such that $p_0(\va)$ is the characteristic polynomial of $\pi^0$.

\smallskip
\underline{Step 3.} Given the non-constant polynomial $p_i(\va)$ for $i \geq 0$, find the maximal $t \geq 1$ such that $p_i(\va)$ is a multiple of $(\va^t-1)$, and define $p_{i+1}(\va)$ such that $p_i(\va)=p_{i+1}(\va)(\va^t-1)$. By Remark~\ref{A:r}, if $p_i(\va)$ is the characteristic polynomial of the partition $\pi^i=(\pi^i_1,\ldots,\pi^i_{\ell_i})$, then $p_{i+1}(\va)$ is the characteristic polynomial of the partition $\pi^{i+1}:=(\pi^i_2,\ldots,\pi^i_{\ell_i})$.  Append the integer $t=\pi^i_1$ to the list $\ct(q)$.

\smallskip
\underline{Step 4.} Starting with the polynomial $p_0(\va)$ of Step~2, repeat Step~3 until we find a constant polynomial $p_{\ell}(\va)\equiv 1$. We end up with a list $\ct(q)$ with $\ell$ elements, which is the wanted cycle type of $q$ by Remark~\ref{A:r}.
\end{algorithm}\normalsize

We close our discussion with a procedure to explicitly find all partitions of fixed length (Algorithm~\ref{A:three}). Using the Main Theorem of the paper, we may find in this way all Coxeter polynomials among connected non-negative unit forms of Dynkin type $\A_{m-1}$ (Algorithm~\ref{A:four}). For the sake of readability, partitions of the integer $m \geq 1$ will be called simply $m$-partitions.

\begin{algorithm}\small \label{A:three}
We describe an implementable function \texttt{partitions\_by\_length$(m,\ell)$} that recursively constructs all $m$-partitions of fixed length $\ell$.

\smallskip
\textbf{Input:} Integers $m \geq 1$ and $\ell \geq 1$.

\smallskip
\textbf{Output:} A (possibly empty) set $\mathbf{P}$ containing all $m$-partitions of length $\ell$.

\smallskip
 \underline{Step 1.} If $1<\ell<m$, consider the result $\mathbf{P}'$ of \texttt{partitions\_by\_length$(m-1,\ell-1)$}, and take
\[
\mathbf{P}_1=\{\text{$(\pi'_1,\ldots,\pi'_{\ell-1},1)$ such that $(\pi'_1\ldots,\pi'_{\ell-1}) \in \mathbf{P}'$}\}.
\]
 Clearly, $\mathbf{P}_1$ is the set of all $m$-partitions $(\pi_1,\ldots,\pi_{\ell})$ of length $\ell$ having $\pi_{\ell}=1$.

\smallskip
\underline{Step 2.}  If $1<\ell<m$, consider the result $\mathbf{P}''$ of \texttt{partitions\_by\_length$(m-\ell,\ell)$}, and take
\[
\mathbf{P}_2=\{ \text{$(\pi''_1+1,\ldots,\pi''_{\ell}+1)$ such that $(\pi''_1\ldots,\pi''_{\ell}) \in \mathbf{P}''$}\}.
\]
 Clearly, $\mathbf{P}_2$ is the set of all $m$-partitions $(\pi_1,\ldots,\pi_{\ell})$ of length $\ell$ having $\pi_{\ell}>1$.

\smallskip
\underline{Step 3.} Return
\begin{equation*}
\mathbf{P} = \left\{ \begin{array}{l l}
(m), & \text{if $\ell=1$},\\
\mathbf{P}_1 \cup \mathbf{P}_2, & \text{if $1<\ell<m$}, \\
(1,\ldots,1), & \text{if $\ell=m$},\\
\emptyset, & \text{if $\ell>m$}.\end{array} \right.
\end{equation*}\normalsize
\end{algorithm}

Recall that for any $c \geq 0$ and $m \geq 1$,
\[
\Part_1^c(m)=\{\pi \vdash m \mid 0 \leq c-(\ell(\pi)-1) \equiv 0 \mod 2 \}.
\]

\begin{algorithm}\small \label{A:four}$ $\par

\smallskip
\textbf{Input:} Integers $n \geq 1$ and $0 \leq c < n$.

\smallskip
\textbf{Output:} The set $\mathbf{CP}_{\A}^c(n)$ of all Coxeter polynomials among the quadratic unit forms in $\Quad_{\A}^c(n)$.

\smallskip
\underline{Step 1.} Take $m=n-c+1$ and consider the set $\mathcal{L}=\{\ell \geq 1 \mid 0 \leq c-(\ell-1) \equiv 0 \mod 2 \}$.

\eject
\underline{Step 2.} For any $\ell \in \mathcal{L}$ take $\mathbf{P}_{\ell}$ the result of the function \texttt{partitions\_by\_length$(m,\ell)$} constructed in Algorithm~\ref{A:three}, and
take the (disjoint) union
\[
\mathbf{P}= \bigcup_{\ell \in \mathcal{L}} \mathbf{P}_{\ell}.
\]

\underline{Step 3.} For any partition $\pi=(\pi_1,\ldots,\pi_{\ell})$ in $\mathbf{P}$, take the polynomial
\[
\chr_{\pi}(\va)=\prod_{a=1}^{\ell}(\va^{\pi_a}-1),
\]
and consider the set $\mathbf{CP}_{\A}^c(n)=\{(\va-1)^{c-1}\chr_{\pi}(\va) \mid \pi \in \mathbf{P}\}$. By Theorem~\ref{MAIN}, we have
\[
\mathbf{CP}_{\A}^c(n)=\{ \cox_q(\va) \mid q \in \Quad_{\A}^c(n) \},
\]
where $\cox_q(\va)$ denotes the Coxeter polynomial of a quadratic form $q$.
\end{algorithm}\normalsize

\begin{remark}\label{A:r2}
Observe that, as consequence of Algorithm~\ref{A:two}, the sets $\mathbf{P}$ and $\mathbf{CP}_{\A}^c(n)$ constructed in Algorithm~\ref{A:four} have the same cardinality. That is, the number of Coxeter polynomials appearing among connected non-negative unit forms in $n$-variables, of Dynkin type $\A_{m-1}$ and corank $c$, is the number of partitions of the integer $m=n-c+1$ whose lengths $\ell$ satisfy $0 \leq c-(\ell-1) \equiv 0 \mod 2$.
\end{remark}

\subsection*{Acknowledgments}

I would like to express my thanks to J.A. de la Pe\~na for many academical discussions, to the Instituto de Matem\'aticas UNAM, Mexico, for a postdoctoral grant, and to the anonymous referees for their careful revisions and useful suggestions.

\bibliographystyle{fundam}

\begin{thebibliography}{10}
\providecommand{\url}[1]{\texttt{#1}}
\providecommand{\urlprefix}{URL }
\expandafter\ifx\csname urlstyle\endcsname\relax
  \providecommand{\doi}[1]{doi:\discretionary{}{}{}#1}\else
  \providecommand{\doi}{doi:\discretionary{}{}{}\begingroup
  \urlstyle{rm}\Url}\fi
\providecommand{\eprint}[2][]{\url{#2}}

\bibitem{BGZ06} Barot M, Geiss C, and Zelevinsky A. \textit{Cluster algebras of finite type and positive symmetrizable matrices}. J. London Math. Soc., 2006. \textbf{73}:545--564. doi:10.1112/S0024610706022769.

\bibitem{BJP19} Barot M, Jim\'enez Gonz\'alez JA,  and de la Pe\~na JA.
\textit{Quadratic Forms: Combinatorics and Numerical Results},
Algebra and Applications, Vol. \textbf{25} Springer Nature Switzerland AG 2018.
ISBN-13:978-3030056261, 10:3030056260.

\bibitem{BKL06} Barot M, Kussin D, and Lenzing H.
\textit{The Lie algebra associated to a unit form}.
Journal of Algebra. 2006. \textbf{296}:1--17. doi:10.1016/j.jalgebra.2005.11.017.

\bibitem{BP99} Barot M, and de la Pe\~na JA.
\textit{The Dynkin type of a non-negative unit form}, Expo. Math. 1999. 17:339--348.

\bibitem{BR07} Barot M, and Rivera D.
\textit{Generalized Serre relations for Lie algebras associated to positive unit forms},
Journal of Pure and Applied Algebra. 2007.  \textbf{211}:360--373. doi:10.1016/j.jpaa.2007.01.008.

\bibitem{BFS14}
Bocian R, Felisiak M, and Simson D.
 \textit{Numeric and mesh algorithms for the Coxeter spectral study of positive edge-bipartite graphs and their isotropy groups}. Journal of Computational and Applied Mathematics. 2014. \textbf{259}:815--827. doi:10.1016/j.cam.2013.07.013.

\bibitem{kB83} Bongartz K.
\textit{Algebras and quadratic forms}. J. London Math. Soc. (2)  1983. \textbf{28}(3):461--469.
doi:10.1112/jlms/s2-28.3.461..

\bibitem{BPS11} Br{\"u}stle T, de la Pe\~na JA, and Skowro\'nski A.
\textit{Tame algebras and Tits quadratic forms},
Advances in Mathematics. 2011.  \textbf{226}:887--951. doi:10.1016/j.aim.2010.07.007.

\bibitem{CDD21}
Cavaleri M, D'Angelo D, and Donno A.
\textit{Characterizations of line graphs in signed and gain graphs}. 2021.
 arXiv:2101.09677v2 [math.CO]

\bibitem{FS13b}
Felisiak M,  and Simson D.
 \textit{On combinatorial algorithms computing mesh root systems and matrix morsifications for the Dynkin diagram $An$.}
Discrete Mathematics. 2013.  313(12):1358--1367. doi:10.1016/j.disc.2013.02.003.

\bibitem{GSZ14}
G\k{a}siorek M, Simson D, and Zaj\k{a}c K.
\textit{Algorithmic computation of principal posets using Maple and Python}.
Algebra and Discrete Math. 2014. \textbf{17}:33--69. doi:10.3233/FI-2016-1345.

\bibitem{jaJ2018} Jim\'enez Gonz\'alez JA.
\textit{Incidence graphs and non-negative integral quadratic forms}.
Journal of Algebra. 2018.  \textbf{513}:208--245.  doi:10.016/J.JALGEBRA.2018.07.020.

\bibitem{jaJ2020a} Jim\'enez Gonz\'alez, J.A.
\textit{A graph theoretical framework for the strong Gram classification of non-negative unit forms of Dynkin type $\A_n$}. Fundamenta Informaticae \textbf{184}(1) : 49--82 (2021)

\bibitem{jaJ2020c} Jim\'enez Gonz\'alez JA.
\textit{A Strong Gram classification of non-negative unit forms of Dynkin type $\A_r$}. Preprint.

\bibitem{KS15a} Kasjan S, and Simson D.
\textit{Mesh algorithms for Coxeter spectral classification of Cox-regular edge-bipartite graphs with loops, I. Mesh root systems.}
Fund. Inform. 2015. \textbf{139}(2):153--184.  doi:10.3233/FI-2015-1230.

\bibitem{KS15b} Kasjan S, and Simson D.
\textit{Mesh algorithms for Coxeter spectral classification of Cox-regular edge-bipartite
graphs with loops, II. Application to Coxeter spectral analysis.}
Fund. Inform. 2015. \textbf{139}(2):185--209. doi:10.3233/FI-2015-1231.

\bibitem{KS15c} Kasjan S, and Simson D.
\textit{Algorithms for Isotropy Groups of Cox-regular Edge-bipartite Graphs}.
Fund. Inform. 2015. \textbf{139}(3):249--275. doi:10.3233/FI-2015-1234.

\bibitem{Pra01}
Prasolov V.
\textit{Polynomials}. Algorithms and computation in Mathematics Vol.~\textbf{11}.
Springer Verlag, Berlin 2001 second edition.  ISBN-13:978-3642039799, 10:3642039790.

\bibitem{cmR} Ringel CM.
 \textit{Tame algebras and integral quadratic forms}, Springer LNM, 1984. \textbf{1099}.
 ISBN-13:978-0387139050, 10:0387139052.

\bibitem{jR06} Rotman JJ.
 \textit{A first course in abstract algebra with applications},
3rd Ed., Upper Saddle River, New Jersey 07458, Pearson Prentice Hall, 2006. ISBN:0-13-186267-7.

\bibitem{mS05} Sato M.
\textit{Periodic Coxeter matrices and their associated quadratic forms}.
Linear Algebra Appl. 2005. \textbf{406}:99--108. doi:10.1016/j.laa.2005.03.036.

\bibitem{dS11a} Simson D. \textit{Mesh geometries of root orbits of integral quadratic forms},
J. Pure Appl. Algebras 2011. \textbf{215}(1):13--34.  doi:10.1016/j.jpaa.2010.02.029.

\bibitem{dS13} Simson D. \textit{A Coxeter Gram classification of positive simply laced edge-bipartite graphs}.
SIAM J. Discrete Math., 2013.  \textbf{27}(2):827--854.   doi:10.1137/110843721.

\bibitem{dS13a} Simson D.
\textit{Algorithms determining matrix morsifications, Weyl orbits, Coxeter polynomials
and mesh geometries of roots for Dynkin diagrams}, Fund. Inform. 2013. \textbf{123}:447--490.
doi:10.3233/FI-2013-820.

\bibitem{dS16a} Simson D.
\textit{Symbolic algorithms computing Gram congruences in the Coxeter spectral classification of
edge-bipartite graphs I. A Gram classification}. Fund. Inform., 2016. \textbf{145}(1):19--48. doi:10.3233/FI-2016-1345.

\bibitem{dS18} Simson D.
\textit{A Coxeter spectral classification of positive edge-bipartite graphs I. Dynkin types
$\mathcal{B}_n$, $\mathcal{C}_n$, $\mathcal{F}_4$, $\mathcal{G}_2$, $\E_6$, $\E_7$, $\E_8$}.
Linear Algebra Appl. 2018. \textbf{557}:105--133.

\bibitem{dS20} Simson D.
\textit{A computational technique in Coxeter spectral study of symmetrizable integer Cartan matrices}.
Linear Algebra Appl., 2020. \textbf{586}:190--238. doi:10.1016/j.laa.2019.10.015.

\bibitem{dS21a} Simson D.
\textit{A Coxeter spectral classification of positive edge-bipartite graphs II. Dynkin type $\D_n$.}
Linear Algebra and its Applications Volume. 2021. \textbf{612}:223--272. doi:10.1016/j.laa.2020.11.001.

\bibitem{dS21b} Simson, D.
\textit{Weyl orbits of matrix morsifications and a Coxeter spectral classification of positive signed graphs and quasi-Cartan matrices of Dynkin type $\A_n$.} Advances in Mathematics Vol.~\textbf{404}, Part A (2022)

\bibitem{SZ12} Simson D, and Zaj\k{a}c K.
\textit{A Framework for Coxeter Spectral Classification of
Finite Posets and Their Mesh Geometries of Roots},
Hindawi Publishing Corporation International Journal of Mathematics and Mathematical
Sciences,  2013, Article ID 743734, 22 pages.  doi:10.1155/2013/743734.

\bibitem{SZ17} Simson D, and Zaj\k{a}c K.
\textit{Inflation algorithm for loop-free non-negative edge-bipartite graphs of corank at least two},
Linear Algebra Appl. 2017. \textbf{524}:109--152.  doi:10.1016/j.laa.2017.02.021.

\bibitem{tZ08} Zaslavsky, T.
\textit{Matrices in the theory of signed simple graphs},
in: Advances in Discrete Mathematics and Applications: Mysore, 2008,
in Ramanujan Math. Soc. Lect. Notes Ser., Vol. \textbf{13}, Ramanujan Math. Soc.,
Mysore, 2010, pp. 207--229.

\bibitem{fZ99} Zhang F.
 \textit{Matrix Theory: Basic Results and Techniques}, 2nd Ed. Springer 1999. ISBN-13:978-0387986968, 10:0387986960.
\end{thebibliography}

\end{document}